\documentclass[11pt]{article}
\usepackage{amsmath}
\usepackage{amsthm}
\usepackage{amssymb}
\usepackage{arydshln}
\usepackage{xcolor}
\usepackage[colorlinks=true,linkcolor={[HTML]{C00020}},citecolor={[HTML]{2000CF}}]{hyperref}
\definecolor{navy}{HTML}{00009F}
\def\a{\mathbf{a}}

\def\b{\mathbf{b}}

\def\Z{\mathbb{Z}}

\def\m{\mathbf{m}}

\renewcommand{\le}{\leqslant}

\newcommand{\coloneqq}{\mathrel{\mathop:}=}
\newcommand{\eqqcolon}{=\mathrel{\mathop:}}
\newtheorem{theorem}{\hspace*{\parindent}Theorem}
\newtheorem{lemma}{\hspace*{\parindent}Lemma}
\newtheorem{corollary}{\hspace*{\parindent}Corollary}
\newtheorem*{corollary*}{\hspace*{\parindent}Corollary}
\DeclareMathOperator*{\sign}{sign}

\title{Trigonometric identities: from Hermite via Meijer, N{\o}rlund and Braaksma to Chu and Johnson and beyond}
\author{Alexander Dyachenko$^{\rm a}$~~and Dmitrii\:Karp$^{\rm b}$\footnote{E-mail: A.\:Dyachenko --  \emph{diachenko@sfedu.ru},  D. Karp -- \emph{dimkrp@gmail.com}}
\\[10pt]
\\
\small{\textit{$\phantom{1}^a$Keldysh Institute of Applied Mathematics, Moscow, Russia}}
\\
\small{\textit{$\phantom{1}^b$Department of Mathematics,  Holon Institute of Technology, Holon, Israel}}
}

\date{}
\begin{document}
\maketitle

\begin{abstract}
Known already to the ancient Greeks, today trigonometric identities come in a large variety of tastes and flavours. In this large family there is a subfamily of interpolation-like identities  discovered by Hermite and revived rather recently in two independent papers, one by Wenchang Chu and the other by Warren Johnson exploring various forms and generalizations of Hermite's results.  The goal of this work inspired by these two articles is twofold. The first goal is to fill a gap in the references from the above papers and  exhibit various trigonometric identities discovered by Meijer, N{\o}rlund and Braaksma between 1940 and 1962 in the context of analytic continuation of Mellin-Barnes integrals and relations between different solutions of the generalized hypergeometric differential equation. Our second goal is to present some extensions of Chu's and Johnson's results by combining them with the ideas of Meijer and Braaksma adding certain sum manipulations and facts from the complex analysis.  We unify and systematize various known and new identities and illustrate our results with numerous explicit examples.
\end{abstract}

\bigskip


\bigskip



\section{Introduction}

In a remarkable paper \cite{Johnson}, Warren Johnson found a partial fraction expansions for the ratio
\begin{equation}\label{eq:JohnsonRatio}
\frac{F(\sin(z), \cos(z))}{\prod_{j=1}^{n}\sin(z-a_j)},    
\end{equation}
where $F$ is a polynomial in two variables of the form 
$$
F(\sin(z),\cos(z))=\sum_{r,s\ge0\atop r+s\le n} t_{r,s}(\sin(z))^r(\cos(z))^s
$$
He showed that it is a far-reaching generalization of the cotangent identity due to Hermite \cite[Theorem~1]{Johnson} and includes a beautiful overview of the early history of the cotangent and many related trigonometric identities.  Johnson's paper appeared in April's 2010 issue of American Mathematical Monthly, but was obviously conceived and started  much earlier and, as noted by the author, "has gone a long way beyond my original conception".  This might explain why the author missed a short paper by Wenchang Chu \cite{Chu} published in AMS Proceedings in January 2008. Chu considers essentially the same problem just using a slightly more general form of the initial function, namely 
\begin{equation}\label{eq:ChuRatio}
\frac{P(e^{iz})}{\prod_{j=1}^{n}\sin(z-a_j)},
\end{equation}
with $P(w)$ being a Laurent polynomial in $w$ consisting of the terms $w^k$ with $|k|\le{n}$.
Chu found a partial fraction decomposition of this ratio and listed a number of particular
cases of his partial fraction decomposition showing how they immediately lead to various
identities discovered in the literature, most notably in \cite{Calogero,Gustafson87,Gustafson94,Mohlenkamp}.  We remark that one can recover Hermite's
cotangent identity \cite[Theorem~1]{Johnson} -- the starting point of Johnson's paper -- by setting $\alpha_j=\beta_j+\pi/2$ in \cite[(8)]{Chu}.  Surprisingly enough, Chu's list of references has empty intersection with that of Johnson! One striking example is the following identity 
\begin{equation}\label{eq:genPtolemy}
\sum\limits_{t=1}^{m}\frac{\prod_{j=1}^{m}\sin(a_j-b_t)}{\prod_{j=1\atop j\ne t}^{m}\sin(b_j-b_t)}=\sin(a_1+\cdots+a_m-b_1-\cdots-b_m),
\end{equation}
which, according to Johnson \cite[(4.2)]{Johnson}, is due to Hermite \cite{Hermite}.  This identity can be viewed as a generalization of Ptolemy's theorem (written in the  trigonometric form): for a quadrilateral inscribed in a circle
the product of the lengths of its diagonals is equal to the sum of the products of the lengths
of the pairs of opposite sides. Chu \cite[(9)]{Chu} attributes the above formula to 
Gustafson~\cite[Lemma~5.10]{Gustafson94}. Jointly with E.G.\:Prilepkina, the second author rediscovered it in \cite[Corollary~3.7]{KPSIGMA2016} in the form of a gamma function identity, to which it reduces in view of Euler's reflection formula $\Gamma(z)\Gamma(1-z)=\pi/\sin(\pi{z})$. 

As it turned out identity~\eqref{eq:ChuRatio} was rediscovered several more times. Both Johnson and Chu missed the references to Meijer~\cite{Meijer}, N{\o}rlund \cite{Norlund} and Braaksma \cite{Braaksma}, where a number of related trigonometric identities have been found in connection with the study of the Mellin-Barnes contour integrals, but only in \cite{Norlund} in an explicit form.
One such identity allowed E.G.\:Prilepkina and the second author to simplify the analytic continuation of a particular case of Meijer's $G$-function in~\cite{KPITSF2023}.
A careful reading of~\cite{Meijer} and~\cite{Braaksma} shows that they in some sense go beyond Chu and Johnson. Namely, the degrees on the numerator polynomials $F$ in \eqref{eq:JohnsonRatio} and $P$ in \eqref{eq:ChuRatio} do not exceed $n$ -- the degree of the denominator  polynomial $\prod_{j=1}^{n}\sin(z-a_j)$, while Meijer and Braaksma found expansions for general degrees. 
These authors were interested in those identities not per se, but rather as one of the tools in deriving representations and expansions for the Mellin-Barnes integrals. The formulas are well hidden behind their complicated notation, they give no explicit examples and keep certain quantities defined implicitly as power series coefficients of some rational functions.

The purpose of this note is twofold. First, in the subsequent Section~2 we revisit Meijer's, N{\o}rlund's and Braaksma's identities explaining their notation and furnishing some explicit examples illustrating the power of their general formulas. We further explain the relation between the results of Meijer and Braaksma and mention their connection to the formulas contained in the seminal paper \cite{Norlund} by N{\o}rlund. Second, we extend the approach of Chu and present various generalizations of all previously mentioned trigonometric identities in Section~3.

\section{Identities of Braaksma and Meijer revisited}

Suppose $\a=(a_1,a_2,\ldots,a_r)$ and $\b=(b_1,b_2,\ldots,b_n)$ are real vectors and write $\varkappa=r-n$ for the difference of their sizes. Define 
\begin{equation}\label{eq:F-defined}
F(z)=\frac{\prod_{j=1}^{r}\sin(a_j-z)}{\prod_{j=1}^{n}\sin(b_j-z)}=\frac{\sin(\a-z)}{\sin(\b-z)},    
\end{equation}
where both here and henceforth $\sin(\a-z)=\prod_{j=1}^{r}\sin(a_j-z)$ and~$\sin(\b-z)=\prod_{j=1}^{n}\sin(b_j-z)$
. We also denote by~$\b_{[j]}$ the vector~$\b$ with its~$j$th component
removed:~$\b_{[j]}=(b_1,\dots,b_{j-1},b_{j+1},\dots,b_n)$.  In this section we will generally follow the notation
of Meijer and Braaksma, except for substitution $\pi{a_j}\to a_j$ and $\pi{b_j}\to b_j$ for all
indices $j$. In order to formulate Braaksma's result we need some preparation in which we will
recover certain details omitted in \cite{Braaksma}. Suppose $r\ge n-1$.  
For sufficiently large positive $\Im(z)$ we  have 
$$
F(z)=\sum_{j=0}^{\infty}C_{j}e^{i\delta_{j}z},
$$
where $\delta_0=-\varkappa$, $\delta_{j}=\delta_0+2{j}$.  Indeed, writing $\sin(z)=e^{iz}(1-e^{-2iz})/(2i)$ and expanding, we will have:
\begin{multline*}
F(z)=\underbrace{(2i)^{-\varkappa}e^{i(\sum{a_j}-\sum{b_j})}}_{C_{0}}e^{i\delta_{0}z}\frac{\prod_{j=1}^{r}\left(1-e^{-2ia_{j}}e^{2iz}\right)}{\prod_{j=1}^{n}\left(1-e^{-2ib_{j}e^{2iz}}\right)}
\\
=C_{0}e^{i\delta_{0}z}\left(A_0+A_1e^{2iz}+A_2e^{4iz}+\cdots+A_re^{2irz}\right)
\left(B_0+B_1e^{2iz}+B_2e^{4iz}+\cdots\right)
\\
=C_{0}e^{i\delta_{0}z}\left(A_0B_0+(A_1+B_1)e^{2iz}+(A_2+A_1B_1+B_2)e^{4iz}+\cdots
+e^{2ijz}\sum_{k+l=j}A_kB_l+\cdots\right),
\end{multline*}
where $A_0=B_0=1$, 
\begin{align*}
A_k
&=(-1)^ke_k\left(e^{-2ia_1},\ldots,e^{-2ia_r}\right)
=
(-1)^k\sum\nolimits_{0\le j_1< \cdots < j_k\le r} e^{-2i(a_{j_1}+\dots+a_{j_k})},
\\
B_k
&=h_k\left(e^{-2ib_1},\ldots,e^{-2ib_n}\right)
=
\sum\nolimits_{0\le j_1\le \cdots \le j_k\le n} e^{-2i(b_{j_1}+\dots+b_{j_k})},
\end{align*}
so that $A_k=0$ for $k>r$. Here, $e_k$ stands for $k$-th elementary symmetric polynomial and $h_k$ is the complete homogeneous symmetric polynomial~\cite[section~3.1.1]{Prasolov}. Hence, we get 
\begin{equation}\label{eq:Cj-defined}
C_j=C_0\sum_{k+l=j}A_kB_l=C_0\sum_{k+l=j}(-1)^ke_k\left(e^{-2ia_1},\ldots,e^{-2ia_r}\right)h_l\left(e^{-2ib_1},\ldots,e^{-2ib_n}\right).
\end{equation}
Similarly, for sufficiently large negative $\Im(z)$ one can check that 
$$
F(z)=\sum_{j=-\infty}^{\varkappa}D_{j}e^{i\delta_{j}z},
$$
where $\delta_{\varkappa}=-\delta_0=\delta_0+2\varkappa=\varkappa$,  $\delta_{j}=\delta_0+2{j}=-\delta_{\varkappa}+2{j}$.
To compute the coefficients we expand using $\sin(z)=-e^{-iz}(1-e^{2iz})/(2i)$:
\begin{multline*}
F(z)=\underbrace{(-2i)^{-\varkappa}e^{i(\sum{b_j}-\sum{a_j})}}_{D_{\varkappa}}e^{i\delta_{\varkappa}z}\frac{\prod_{j=1}^{r}\left(1-e^{2ia_{j}}e^{-2iz}\right)}{\prod_{j=1}^{n}\left(1-e^{2ib_{j}e^{-2iz}}\right)}
\\
=D_{\varkappa}e^{i\delta_{\varkappa}z}\left(\hat{A}_0+\hat{A}_1e^{-2iz}+\hat{A}_2e^{-4iz}+\cdots+\hat{A}_re^{-2irz}\right)
\left(\hat{B}_0+\hat{B}_1e^{-2iz}+\hat{B}_2e^{-4iz}+\cdots\right)
\\
=D_{\varkappa}e^{i\delta_{\varkappa}z}
\left(\hat{A}_0\hat{B}_0+(\hat{A}_1+\hat{B}_1)e^{-2iz}+(\hat{A}_2+\hat{A}_1\hat{B}_1+\hat{B}_2)e^{-4iz}+\cdots
+e^{-2imz}\sum_{k+l=m}\hat{A}_k\hat{B}_l+\cdots\right),
\end{multline*}
where $\hat{A}_0=\hat{B}_0=1$, 
$$
\hat{A}_k=(-1)^ke_k\left(e^{2ia_1},\ldots,e^{2ia_r}\right),~~~
\hat{B}_k=h_k\left(e^{2ib_1},\ldots,e^{2ib_n}\right)
$$
and $\hat{A}_k=0$ for $k>r$. Hence,
\begin{equation}\label{eq:Dj-defined}
D_{\varkappa-j}
=D_{\varkappa}\sum_{k+l=\varkappa-j}\hat{A}_k\hat{B}_l
=D_{\varkappa}\sum_{k+l=\varkappa-j}(-1)^ke_k\left(e^{2ia_1},\ldots,e^{2ia_r}\right)h_l\left(e^{2ib_1},\ldots,e^{2ib_n}\right).
\end{equation}
We take a closer look at coefficients similar to~$C_j$ and~$D_j$ in our Corollary~\ref{cr:sine.over.sine} in subsection~\ref{sec:appl.to.trig} below.
The following result is a part of \cite[Lemma~10]{Braaksma} (after correcting an error in \cite[(11.14)]{Braaksma} and a slight change of notation $D_j\to-D_{j}$):
\begin{theorem}\label{th:Braaksma}
Suppose $k\in\Z$ is arbitrary and the numbers $b_1,\ldots,b_n$ are distinct modulo $\pi$. Then for $F$ defined in \eqref{eq:F-defined} and all complex $z$ except for poles of the functions involved:
\begin{equation}\label{eq:mainBraaksma}
F(z)=\sum_{j=0}^{k-1}C_je^{i\delta_{j}z}+\sum_{j=k}^{\varkappa}D_je^{i\delta_{j}z}+\sum_{j=1}^{n}\frac{\sin(\a-b_j)e^{i(\varkappa+1-2k)(b_j-z)}}{\sin(\b_{[j]}-b_{j})\sin(b_j-z)}.
\end{equation}
The sums with the lower limit greater than the upper limit are assumed to equal zero. The numbers $C_j$, $D_j$ are defined in \eqref{eq:Cj-defined} and \eqref{eq:Dj-defined}, respectively.
\end{theorem}

Some particular cases are as follows. Recall that $\a=(a_1,\ldots,a_r)$, $\b=(b_1,\ldots,b_n)$ and denote 
\begin{equation}\label{eq:nu-defined}
\nu=\sum\nolimits_{j=1}^{r}{a_j}-\sum\nolimits_{j=1}^{n}{b_j}.    
\end{equation}
If $\varkappa=r-n=-1$ and $k=0$ we simply have  
$$
F(z)=\sum_{j=1}^{n}\frac{\sin(\a-b_j)}{\sin(\b_{[j]}-b_{j})\sin(b_j-z)}.
$$
This formula was earlier discovered by Meijer (see details below in this section) and N{\o}rlund \cite[(3.32)]{Norlund}.
If  $\varkappa=r-n=0$ and $k=0$ we obtain
$$
F(z)=e^{-i\nu}+\sum_{j=1}^{n}\frac{\sin(\a-b_j)e^{i(b_j-z)}}{\sin(\b_{[j]}-b_{j})\sin(b_j-z)}
$$
which was also earlier found by N{\o}rlund \cite[(3.41)]{Norlund}.
If  $\varkappa=r-n=1$ and $k=1$ we have 
\begin{equation}\label{eq:BraaksmaKappa1}
F(z)=\sin(\nu-z)
+\sum_{j=1}^{n}\frac{\sin(\a-b_j)}{\sin(\b_{[j]}-b_{j})\sin(b_j-z)}.
\end{equation}

Note further that for any odd $\varkappa$ and $k=(\varkappa+1)/2$ we can pair the terms in 
$$
\sum_{j=0}^{k-1}C_je^{i\delta_{j}z}+\sum_{j=k}^{\varkappa}D_je^{i\delta_{j}z}
$$
to get a sum of sine functions.  Indeed, $\delta_{\varkappa}=-\delta_{0}=\varkappa$, $\delta_{\varkappa-1}=-\delta_1=\varkappa-2$, $\ldots$,
$\delta_{k}=-\delta_{k-1}=1$ and 
$$
(2i)^{\varkappa}C_0=e^{i\nu}=-\frac{1}{(2i)^{\varkappa}D_{\varkappa}},
$$
(recall that $\varkappa$ is odd).  This yields
$$
C_0e^{i\delta_{0}z}+D_{\varkappa}e^{i\delta_{\varkappa}z}=\frac{1}{(2i)^{\varkappa-1}}\sin\big(\nu+\delta_{0}z\big).
$$
Next,
\begin{equation*}
\begin{split}
(2i)^{\varkappa}C_1=&e^{i\nu}\left(e^{-2i{b_1}}+\dots+e^{-2i{b_n}}-e^{-2i{a_1}}-\cdots-e^{-2i{a_r}}\right), 
\\
(2i)^{\varkappa}D_{\varkappa-1}=&-e^{-i\nu}\left(e^{2i{b_1}}+\dots+e^{2i{b_n}}-e^{2i{a_1}}-\cdots-e^{2i{a_r}}\right).
\end{split}
\end{equation*}
Hence,
$$
C_1e^{i\delta_{1}z}+D_{\varkappa-1}e^{i\delta_{\varkappa-1}z}=\frac{1}{(2i)^{\varkappa-1}}\sum_{j=1}^{n}\sin\big(\nu-2b_j+\delta_{1}z\big)
-\frac{1}{(2i)^{\varkappa-1}}\sum_{j=1}^{r}\sin\big(\nu-2a_j+\delta_{1}z\big).
$$
Further,
\begin{equation*}
\begin{split}
(2i)^{\varkappa}e^{-i\nu}C_2=&\sum_{1\le{j}<l\le{r}}e^{-2i(a_{j}+a_{l})}
-\sum_{j=1}^{r}\sum_{l=1}^{n}e^{-2i(a_{j}+b_{l})}+\sum_{1\le{j}\le{l}\le{n}}e^{-2i(b_{j}+b_{l})}, 
\\
-(2i)^{\varkappa}e^{i\nu}D_{\varkappa-2}=&\sum_{1\le{j}<l\le{r}}e^{2i(a_{j}+a_{l})}
-\sum_{j=1}^{r}\sum_{l=1}^{n}e^{2i(a_{j}+b_{l})}+\sum_{1\le{j}\le{l}\le{n}}e^{2i(b_{j}+b_{l})}.
\end{split}
\end{equation*}
Hence,
\begin{multline*}
C_2e^{i\delta_{2}z}+D_{\varkappa-2}e^{i\delta_{\varkappa-2}z}=\frac{1}{(2i)^{\varkappa-1}}
\bigg\{\sum_{1\le{j}<l\le{r}}\sin\left(\nu-2a_{j}-2a_{l}+\delta_2z\right)
\\
-\sum_{j=1}^{r}\sum_{l=1}^{n}\sin\left(\nu-2a_{j}-2b_{l}+\delta_2z\right)
+\sum_{1\le{j}\le{l}\le{n}}\sin\left(\nu-2b_{j}-2b_{l}+\delta_2z\right)\bigg\}.
\end{multline*}
In general, each monomial in
$$
(2i)^{\varkappa}C_j=e^{i\nu}\sum_{k+l=j}(-1)^ke_k\left(e^{-2ia_1},\ldots,e^{-2ia_r}\right)h_l\left(e^{-2ib_1},\ldots,e^{-2ib_n}\right)
$$
is of the form $e^{i\alpha}$ and has the counterpart of the form $-e^{-i\alpha}$ in 
$$
(2i)^{\varkappa}D_{\varkappa-j}=-e^{-i\nu}\sum_{k+l=j}(-1)^ke_k\left(e^{2ia_1},\ldots,e^{2ia_r}\right)h_l\left(e^{2ib_1},\ldots,e^{2ib_n}\right),
$$
so that 
$$
C_je^{i\delta_{j}z}+D_{\varkappa-j}e^{i\delta_{\varkappa-j}z}
$$
is necessarily a finite sum of sine functions. For instance, for $\varkappa=3$, $k=2$ we get
\begin{multline}\label{eq:Fkappa3}
F(z)=-\frac{1}{4}\sin(\nu-3z)-\frac{1}{4}\sum_{j=1}^{n}\sin(\nu-2b_j-z)
+\frac{1}{4}\sum_{j=1}^{r}\sin(\nu-2a_j-z)
\\
+\sum_{j=1}^{n}\frac{\sin(\a-b_j)}{\sin(\b_{[j]}-b_{j})\sin(b_j-z)}.
\end{multline}

Next, we turn our attention to Meijer's identity. According to \cite[Lemma~3]{Meijer} for any integer $k$ we have 
\begin{equation}\label{eq:MeijerLemma3}
(2{i})^{\varkappa+1}\sum\limits_{t=1}^{n}e^{(-\varkappa+2k){i}b_t}\frac{\sin(\a-b_t)}{\sin(\b_{[t]}-b_t)}
=e^{\nu{i}}\bar{\Omega}_{\varkappa-k}-(-1)^{\varkappa}e^{-\nu{i}}\Omega_k,
\end{equation}
where the numbers $\Omega_m=\Omega_m(\a,\b)$ and $\bar{\Omega}_m=\bar{\Omega}_m(\a,\b)$ are found from the power series expansions:
$$
\frac{\prod_{j=1}^{r}(1-ze^{2{i}a_j})}{\prod_{j=1}^{n}(1-ze^{2{i}b_j})}=\sum\limits_{m}\Omega_mz^m
\quad\text{and}\quad
\frac{\prod_{j=1}^{r}(1-ze^{-2{i}a_j})}{\prod_{j=1}^{n}(1-ze^{-2{i}b_j})}=\sum\limits_{m}\bar{\Omega}_mz^m.
$$
Note that $\Omega_m=\bar{\Omega}_m=0$ for $m<0$. Moreover, comparison with~\eqref{eq:Cj-defined} and~\eqref{eq:Dj-defined} shows that
\begin{equation*}
\Omega_m = \frac{D_{\varkappa-m}}{D_\varkappa}
\quad\text{and}\quad
\bar \Omega_m = \frac{C_m}{C_0}.
\end{equation*}

It is not hard to observe a relation between Braaksma's identity from Theorem~\ref{th:Braaksma} and Mejer's identity~\eqref{eq:MeijerLemma3}.
On shifting~$n\mapsto n-1$  (so that $\varkappa\mapsto\varkappa+1$), changing~$k\mapsto \varkappa+1-k$ and letting~$z=b_n$, Braaksma's identity turns into
\begin{equation} \label{eq:Braaksma.to.Meijer}
    \frac{\sin(\a -b_n)}{\prod_{j=1}^{n-1}\sin(b_n-b_j)}+
    \sum_{j=1}^{n-1}\frac{\sin(\a-b_j)e^{i(\varkappa+2-2\varkappa-2+2k)(b_j-b_n)}}{\prod_{\substack{m=1\\m\ne j}}^{n-1}\sin(b_{m}-b_{j})\cdot\sin(b_n-b_j)}
    =
    \sum_{j=0}^{\varkappa-k}\widetilde C_je^{i\widetilde\delta_{j}b_n}
    +\!\!\sum_{j=\varkappa+1-k}^{\varkappa+1}\!\!\widetilde D_je^{i\widetilde\delta_{j}b_n}
\end{equation}
for~$\widetilde C_j$ and $\widetilde D_j$ equal resp.\ to~$C_j$ and $D_j$ with shifted
parameters, and for~$\widetilde\delta_0=-\varkappa-1$,
$\widetilde\delta_{j}=\widetilde\delta_0+2{j}$. 
At the same time, 
\[
\sum_{l=0}^\infty x^l h_{l}(\gamma_1,\dots,\gamma_n)
=(1-x \gamma_n)^{-1}\prod_{l=1}^{n-1}(1-x \gamma_l)^{-1}
=\sum_{l=0}^\infty x^l\sum_{j=0}^l h_{l-j}\big(\gamma_1,\dots,\gamma_{n-1}\big)\gamma_n^j,
\]
and hence
$$
h_l\left(e^{-2ib_1},\ldots,e^{-2ib_n}\right)%
=\sum_{j=0}^l h_{l-j}\left(e^{-2ib_1},\dots,e^{-2ib_{n-1}}\right)e^{-2i j b_n}.
$$
On plugging this identity along with~$\widetilde C_0^{-1} e^{ib_n}=(2i)^{\varkappa+1}e^{-i(\sum{a_j}-\sum_{j<n}{b_j})}e^{ib_n}=(2i)^{\varkappa+1}e^{-i\nu}$
into~\eqref{eq:Cj-defined}, we obtain
\[
\begin{aligned}
    \bar\Omega_k&=
    \frac{C_k}{C_0}
    =\sum_{l=0}^k (-1)^{k-l}e_{k-l}\left(e^{-2ia_1},\ldots,e^{-2ia_r}\right)
    \sum_{j=0}^l h_{l-j}\left(e^{-2ib_1},\dots,e^{-2ib_{n-1}}\right)e^{-2i j b_n}
    \\
    &=\sum\nolimits_{j=0}^k e^{-2i j b_n}
    \sum\nolimits_{l=j}^k (-1)^{k-l}e_{k-l}\left(e^{-2ia_1},\ldots,e^{-2ia_r}\right)
    h_{l-j}\left(e^{-2ib_1},\dots,e^{-2ib_{n-1}}\right)
    \\
    &=\sum\nolimits_{j=0}^k e^{(-2j) i b_n}{\widetilde C_{k-j}}{\widetilde C_0^{-1}}
    ={\widetilde C_0^{-1}}
    \sum\nolimits_{j=0}^k e^{(\varkappa+1+\widetilde\delta_j-2k) i b_n}\widetilde C_{j}
    \\
    &
    =
    (2i)^{\varkappa+1} e^{-i\nu}
    e^{i (\varkappa-2k) b_n}
    \sum\nolimits_{j=0}^k e^{i \widetilde\delta_j b_n}\widetilde C_{j}
\end{aligned}
\]
and, consequently,
\[
    e^{i\nu}\bar\Omega_{\varkappa-k}
    =
    (2i)^{\varkappa+1}
    e^{i (-\varkappa+2k) b_n}
    \sum\nolimits_{j=0}^{\varkappa-k} e^{i \widetilde\delta_j b_n}\widetilde C_{j}
.
\]
From~\eqref{eq:Dj-defined} we similarly have
\[
    \Omega_k
    =
    \frac{D_{\varkappa-k}}{D_\varkappa}
    =\sum_{j=0}^k e^{(2j) i b_n}\frac{\widetilde D_{\varkappa+1-k+j}}{\widetilde D_{\varkappa+1}}
    =\widetilde D_{\varkappa+1}^{-1}
    \sum\nolimits_{j=\varkappa+1-k}^{\varkappa+1} e^{(2j-2\varkappa-2+2k) i b_n}\widetilde D_{j},
\]
and thus, due to~$\widetilde D_{\varkappa+1}^{-1} e^{-ib_n}=(-2i)^{\varkappa+1}e^{i(\sum{a_j}-\sum_{j<n}{b_j})}e^{ib_n}=(-2i)^{\varkappa+1}e^{i\nu}$,
\[
    -(-1)^{\varkappa}e^{-\nu{i}}\Omega_k=(2i)^{\varkappa+1}
    \!\!\sum_{j=\varkappa+1-k}^{\varkappa+1}\!\! e^{(\widetilde\delta_j+\varkappa+1-2\varkappa-1+2k) i b_n}\widetilde D_{j}
    =(2i)^{\varkappa+1}e^{(-\varkappa+2k) i b_n}
    \!\!\sum_{j=\varkappa+1-k}^{\varkappa+1}\!\! e^{i \widetilde\delta_j b_n}\widetilde D_{j}
    .
\]
The terms on the left-hand side of the
equality~\eqref{eq:Braaksma.to.Meijer} may be combined into one sum, so~\eqref{eq:Braaksma.to.Meijer} times~$(2{i})^{\varkappa+1} e^{(-\varkappa+2k){i}b_n}$ is exactly the same as~\eqref{eq:MeijerLemma3}:
\begin{align*}
    (2{i})^{\varkappa+1}
    \sum\limits_{t=1}^{n}e^{(-\varkappa+2k){i}b_t}
    \frac{\sin(\a-b_t)}{\sin(\b_{[t]}-b_t)}
    &=
    (2{i})^{\varkappa+1}
    e^{(-\varkappa+2k){i}b_n}
    \left(\sum_{j=0}^{k} \widetilde C_j e^{i\widetilde\delta_{j}b_n}
        +\sum_{j=k+1}^{\varkappa+1}\widetilde D_j e^{i\widetilde\delta_{j}b_n}
    \right)
    \\
    &=\left(e^{\nu{i}}\bar{\Omega}_{\varkappa-k}-(-1)^{\varkappa}e^{-\nu{i}}\Omega_k\right)
    .
\end{align*}

In particular, for $\varkappa=k=0$  formula \eqref{eq:MeijerLemma3} yields \eqref{eq:genPtolemy}.
If $\varkappa=1$ ($r=n+1$) and $k=0$ we obtain
$$
\sum\limits_{t=1}^{n}e^{-{i}b_t}
\frac{\sin(\a-b_t)}{\sin(\b_{[t]}-b_t)}=\frac{1}{4}e^{\nu{i}}\left(\sum\nolimits_{j=1}^{n+1}e^{-2{i}a_j}-\sum\nolimits_{j=1}^{n}e^{-2{i}b_j}\right)-\frac{1}{4}e^{-\nu{i}}.
$$
If $\varkappa=2$ ($r=n+2$) and $k=1$, then
\begin{multline*}
(2i)^3\sum\limits_{t=1}^{n}\frac{\sin(\a-b_t)}{\sin(\b_{[t]}-b_t)}=e^{\nu{i}}\left(\sum\nolimits_{j=1}^{n}e^{-2{i}b_j}-\sum\nolimits_{j=1}^{n+2}e^{-2{i}a_j}\right)
\\
-e^{-\nu{i}}\Big(\sum\nolimits_{j=1}^{n}e^{2{i}b_j}-\sum\nolimits_{j=1}^{n+2}e^{2{i}a_j}\Big)
=2i\Big(\sum\nolimits_{j=1}^{n}\sin(\nu-b_j)-\sum\nolimits_{j=1}^{n+2}\sin(\nu-a_j)\Big)
\end{multline*}
or
$$
\sum\limits_{t=1}^{n}\frac{\sin(\a-b_t)}{\sin(\b_{[t]}-b_t)}
=\frac{1}{4}\sum\nolimits_{j=1}^{n+2}\sin(\nu-a_j)-\frac{1}{4}\sum\nolimits_{j=1}^{n}\sin(\nu-b_j).
$$

A corollary of \eqref{eq:MeijerLemma3} stated as \cite[Lemma~4]{Meijer} and, in a slightly different notation, also as \cite[Lemma~5]{Meijer}, asserts that for $r\ge0$, $m\ge1$, $\varkappa\le-1$ and $0\le k\le -\varkappa-1$ we have
\begin{equation}\label{eq:MeijerLemma4}
\sum\limits_{t=1}^{n}e^{(\varkappa+2k+1){i}(z-b_t)}\frac{\sin(\a-b_t)}{\sin(\b_{[t]}-b_t)\sin(z-b_{t})}
=-\frac{\sin(\a-z)}{\sin(\b-z)}
\end{equation}
under the condition that $z-b_t,\;b_j-b_t\notin\Z$ for all $j\ne t$. On substituting~$k\mapsto k-\varkappa-1$, the last formula reduces to a particular case of Theorem~\ref{th:Braaksma}.

In Theorem~\ref{th:Braaksma} and in the resulting identities we require that the numbers $b_1,\ldots,b_n$ be distinct modulo $\pi$.  If there is a subset of equal (or equal modulo $\pi$) values in this set, N{\o}rlund \cite[(3.36) and (3.49)]{Norlund} derived generalizations of some of the above identities, for the case of non-positive $\varkappa$.

\section{Extending the formula of Chu}

\subsection{Partial fraction decomposition}
The main expansion due to Chu \cite[p.231]{Chu} is given by
\begin{equation}\label{eq:ChuAsis}
\frac{P(w)w^{n+1}}{\prod_{k=0}^{n}(w^2-\gamma_k^2)}=c_{n+1}
+\frac{1}{2}\sum\limits_{k=0}^{n}\frac{\gamma_k^n}{\prod_{j\ne{k}}(\gamma_k^2-\gamma_j^2)}\left(\frac{P(\gamma_k)}{w-\gamma_k}+(-1)^n\frac{P(-\gamma_k)}{w+\gamma_k}\right),
\end{equation}
where $P(w)=c_{n+1}w^{n+1}+c_nw^n+\cdots+c_{-n-1}w^{-n-1}$ is a Laurent polynomial and the numbers $\gamma_k$ are distinct. Following Chu \cite[p.231]{Chu} we can write (with $w=e^{iz}$ and $\gamma_k=e^{i\beta_k}$):
$$
R(z)=\frac{P(e^{iz})}{\prod_{k=0}^{n}\sin(z-\beta_k)}=(2i)^{n+1}e^{\mathcal{B}i}\frac{P(w)w^{n+1}}{\prod_{k=0}^{n}(w^2-\gamma_k^2)},
$$
where $\mathcal{B}=\sum_{k=0}^{n}\beta_k$ and this formula holds for $P(w)$ of arbitrary degree.  Combining this with~\eqref{eq:ChuAsis}, Chu obtained an expansion for $R(z)$. Note, however, that formula \eqref{eq:ChuAsis} holds if the degree of the polynomial $P(w)w^{n+1}$ does not exceed $2n+2$ which is the degree of the polynomial in the denominator. Our goal below is to extend Chu's identity to the Laurent polynomial $P(w)$ of arbitrary degree. To this end we begin with the monomial case elaborated in the following lemma. 
\begin{lemma}
    Given two integers~$m,n\ge 0$ take any integer $\vartheta\le m+n$ having the same parity as~$m+n$. Then the following expansion holds:
    \begin{equation}\label{eq:w.basic.expression}
\frac{w^{m+n}}{\prod_{k=0}^n(w^2-\gamma_k^2)}
=
\sum_{j=0}^{(m-n-\vartheta)/2-1}
h_{j}(\gamma_0^2,\dots,\gamma_n^2) w^{m-n-2-2j}+
\sum_{k=0}^n\frac{\gamma_k^{m+n-\vartheta}w^\vartheta}{\prod_{\substack{j=0\\j\ne k}}^{n}(\gamma_k^2-\gamma_j^2)(w^2-\gamma_k^2)},
\end{equation}
   where~$h_j$ denotes the~$j$-th complete homogeneous symmetric polynomial.
\end{lemma}
\begin{proof}
On the one hand, given an integer~$m\ge 0$ the formula
\[
\frac{w^{m+n}}{\prod_{k=0}^n(w^2-\gamma_k^2)}
=
\frac{w^{m-n-2}}{\prod_{k=0}^n\big(1-\gamma_k^2/w^2\big)}
=
w^{m-n-2}
\sum_{j=0}^{\infty}
h_j(\gamma_0^2,\dots,\gamma_n^2){w^{-2j}}
\]
holds whenever~$|w|$ is sufficiently large. On the other hand, if~$m+n-\vartheta$ is a nonnegative even number, the difference
\[
\frac{w^{m+n-\vartheta}}{\prod_{k=0}^n(w^2-\gamma_k^2)}
-
\sum_{k=0}^n\frac{\gamma_k^{m+n-\vartheta}}{\prod_{\substack{j=0\\j\ne k}}^{n}(\gamma_k^2-\gamma_j^2)(w^2-\gamma_k^2)}
\]
is analytic in the whole complex plane and hence,  Liouville's theorem implies that
\[
\frac{w^{m+n-\vartheta}}{\prod_{k=0}^n(w^2-\gamma_k^2)}
=
\sum_{j=0}^{(m-n-\vartheta)/2-1}
h_{j}(\gamma_0^2,\dots,\gamma_n^2) w^{(m-n-\vartheta)-2-2j}+
\sum_{k=0}^n\frac{\gamma_k^{m+n-\vartheta}}{\prod_{\substack{j=0\\j\ne k}}^{n}(\gamma_k^2-\gamma_j^2)(w^2-\gamma_k^2)}
\]
assuming that the first sum on the right-hand side is zero if~$m-n-\vartheta<2$. The last
equality multiplied by~$w^{\vartheta}$ yields the desired formula.
\end{proof}

In applications of this lemma we will frequently take ~$\vartheta\in\{-1,0,1\}$ of the same
parity as~$m+n$, which corresponds to~$\vartheta=(m+n)-2\big\lfloor\frac{m+n}2\big\rfloor$
or~$\vartheta=2\big\lfloor\frac{m+n}2\big\rfloor-(m+n)$. 
Then, if~$\vartheta\in\{0,1\}$ we have
\[
2\frac{\gamma_k^{1-\vartheta}w^\vartheta}{w^2-\gamma_k^2}
=\frac{w+\gamma_k +
(-1)^{1-\vartheta}({w-\gamma_k})}{w^2-\gamma_k^2}=
\frac{1}{w-\gamma_k}+
\frac{(-1)^{1-\vartheta}}{w+\gamma_k}
\quad\text{and}\quad
(-1)^{1-\vartheta}=(-1)^{m+n-1}
,
\]
and for~$\vartheta\in\{0,-1\}$ we have
\[
2\frac{\gamma_k^{-\vartheta}w^{1+\vartheta}}{w^2-\gamma_k^2}=
\frac{w+\gamma_k +
(-1)^{-\vartheta}({w-\gamma_k})}{w^2-\gamma_k^2}=
\frac{1}{w-\gamma_k}+
\frac{(-1)^{-\vartheta}}{w+\gamma_k}
\quad\text{and}\quad
(-1)^{-\vartheta}=(-1)^{m+n}
.
\]

Multiplying  formula~\eqref{eq:w.basic.expression} by $w$ and choosing the first and the second of the above decompositions, respectively, we arrive at:
\begin{align}\nonumber
\frac{w^{m+n+1}}{\prod_{k=0}^n(w^2-\gamma_k^2)}
=
\sum_{j=0}^{(m-n-|\vartheta|)/2-1}&
h_{j}(\gamma_0^2,\dots,\gamma_n^2) w^{m-n-1-2j}
\\\label{eq:w.basic.sum-}
&+\sum_{k=0}^{n}\frac{w}{2\prod_{\substack{j=0\\j\ne k}}^{n}(\gamma_k^2-\gamma_j^2)}\Big(\frac{\gamma_k^{m+n-1}}{w-\gamma_k}
+
\frac{(-\gamma_k)^{m+n-1}}{w+\gamma_k}\Big)
,
\\\nonumber
\frac{w^{m+n+1}}{\prod_{k=0}^n(w^2-\gamma_k^2)}
=
\!\!\!\!\!\sum_{j=0}^{(m-n+|\vartheta|)/2-1}&
h_{j}(\gamma_0^2,\dots,\gamma_n^2) w^{m-n-1-2j}
\\\label{eq:w.basic.sum+}
&+\sum_{k=0}^n\frac{1}{2\prod_{\substack{j=0\\j\ne k}}^{n}(\gamma_k^2-\gamma_j^2)}\Big(\frac{\gamma_k^{m+n}}{w-\gamma_k}
+
\frac{(-\gamma_k)^{m+n}}{w+\gamma_k}\Big).
\end{align}
After substituting~$w\mapsto w^{-1}$ and~$\gamma_k\mapsto \gamma_k^{-1}$, the expressions~\eqref{eq:w.basic.sum-} and~\eqref{eq:w.basic.sum+} turn respectively into
\begin{align}
\frac{w^{-m+n+1}}{\prod_{k=0}^n(w^2-\gamma_k^2)}
&=\sum_{j=0}^{(m-n-|\vartheta|)/2-1}
\frac{h_{j}(\gamma_0^{-2},\dots,\gamma_n^{-2})}{\prod_{k=0}^n(-\gamma_k^2)} w^{-(m-n-1-2j)}\nonumber
\\\label{eq:w.basic.neg.sum-}
&+\sum_{k=0}^n\frac{1}{2\prod_{\substack{j=0\\j\ne k}}^{n}(\gamma_k^2-\gamma_j^2)}
\Big(\frac{\gamma_k^{-m+n}}{w-\gamma_k}+\frac{(-\gamma_k)^{-m+n}}{w+\gamma_k}\Big),
\end{align}
and
\begin{align}
\frac{w^{-m+n+1}}{\prod_{k=0}^n(w^2-\gamma_k^2)}
&=\!\!\!\!\!\!
\sum_{j=0}^{(m-n+|\vartheta|)/2-1}\!
\frac{h_{j}(\gamma_0^{-2},\dots,\gamma_n^{-2})}{\prod_{k=0}^n(-\gamma_k^2)} w^{-(m-n-1-2j)}\nonumber
\\\label{eq:w.basic.neg.sum+}
&+\sum_{k=0}^n \frac{w}{2\prod_{\substack{j=0\\j\ne k}}^{n}(\gamma_k^2-\gamma_j^2)}
\Big(\frac{\gamma_k^{-m+n-1}}{w-\gamma_k}+\frac{(-\gamma_k)^{-m+n-1}}{w+\gamma_k}\Big).
\end{align}

If now we put 
$$
P_{n+1}(w)=\sum_{m=-n-1}^{n+1} c_m w^{m},
$$
then by taking a linear combination with the coefficients~$c_m$ of~\eqref{eq:w.basic.sum+} and~\eqref{eq:w.basic.neg.sum-} written for~$m=0,\ldots,n+1$, we arrive at Chu's identity~\eqref{eq:ChuAsis}.
In a similar fashion, for $P_{n+2}(w)$ containing degrees from $w^{-n-2}$ to $w^{n+2}$ formulas~\eqref{eq:w.basic.sum+} and~\eqref{eq:w.basic.neg.sum-} applied with~$m=0,\dots,n+2$ imply
\begin{equation}\label{eq:Chuplus1}
\frac{P_{n+2}(w)w^{n+1}}{\prod_{k=0}^{n}(w^2-\gamma_k^2)}=c_{n+2}w+c_{n+1}
-\frac{(-1)^nc_{-n-2}}{w\prod_{k=0}^{n}\gamma_k^2}
+\frac{1}{2}\sum\limits_{k=0}^{n}\frac{\gamma_k^n}{\prod_{j\ne{k}}(\gamma_k^2-\gamma_j^2)}\left(\frac{P(\gamma_k)}{w-\gamma_k}+(-1)^n\frac{P(-\gamma_k)}{w+\gamma_k}\right).
\end{equation}
The sum of residues of a rational function at all finite points is well-known to be equal to its residue at infinity \cite[(4.1.14)]{AF}, which is the coefficient at $w^{-1}$ in the asymptotic expansion ($w\to\infty$) 
\begin{equation}\label{eq:Fasymp}
\frac{P_{n+2}(w)w^{n+1}}{\prod_{k=0}^{n}(w^2-\gamma_k^2)}=\frac{c_{n+2}w^{2n+4}+c_{n+1}w^{2n+3}+\cdots}{w^{2n+3}+\alpha w^{2n+1}+\cdots}
=c_{n+2}w+c_{n+1}+\frac{c_n-\alpha{c_{n+2}}}w+O\Big(\frac{1}{w^2}\Big).
\end{equation}
Here $\alpha=-\sum_{k=}^{n}\gamma_k^2$ in view of 
$$
\prod_{k=0}^{n}(w^2-\gamma_k^2)=w^{2n+2}-w^{2n}\sum_{k=0}^{n}\gamma_k^2+O\big(w^{2n-2}\big)~~\text{as}~w\to\infty.
$$
Equating the coefficient at $w^{-1}$ to the sum of residues at all  finite points leads to the identity 
\begin{equation}\label{eq:sumres1}
c_{n+2}\sum_{k=0}^{n}\gamma_k^2+c_{n}=-\frac{(-1)^nc_{-n-2}}{\prod_{k=0}^{n}\gamma_k^2}+\frac{1}{2}\sum_{k=0}^{n}\frac{\gamma_k^n(P(\gamma_k)+(-1)^nP(-\gamma_k))}{\prod_{j\ne{k}}(\gamma_k^2-\gamma_j^2)}.
\end{equation}

This straightforward approach may be continued. For $P_{n+3}(w)$ by application of \eqref{eq:w.basic.sum+} and~\eqref{eq:w.basic.neg.sum-} for~$m=0,\dots, n+3$ one arrives at the expansion 
\begin{multline}\label{eq:Chuplus2}
 \frac{P_{n+3}(w)w^{n+1}}{\prod_{k=0}^{n}(w^2-\gamma_k^2)}=c_{n+3}w^2+c_{n+2}w+c_{n+1}+c_{n+3}\sum_{k=0}^{n}\gamma_k^2
\\
+\frac{(-1)^{n+1}c_{-n-2}}{w\prod_{k=0}^{n}\gamma_k^2}+\frac{(-1)^{n+1}c_{-n-3}}{w^2\prod_{k=0}^{n}\gamma_k^2}
+\frac{1}{2}\sum\limits_{k=0}^{n}\frac{\gamma_k^n}{\prod_{j\ne{k}}(\gamma_k^2-\gamma_j^2)}\left(\frac{P(\gamma_k)}{w-\gamma_k}+(-1)^n\frac{P(-\gamma_k)}{w+\gamma_k}\right).
\end{multline}
Equating the coefficient at $w^{-1}$ in the  asymptotic approximation 
$$
\frac{P_{n+3}(w)w^{n+1}}{\prod_{k=0}^{n}(w^2-\gamma_k^2)}=c_{n+3}w^2+c_{n+2}w+c_{n+1}+c_{n+3}\sum_{k=0}^{n}\gamma_k^2+\Big(c_{n}+c_{n+2}\sum\nolimits_{k=0}^{n}\gamma_k^2\Big)\frac{1}{w}+O\Big(\frac{1}{w^2}\Big)
$$
as $w\to\infty$, to the sum of residues at all finite points, we again obtain identity \eqref{eq:sumres1}.

It is not hard to write expressions for Laurent polynomials of higher degrees in the numerator, although they become increasingly cumbersome, e.g.
\begin{equation}\label{eq:Chuplus3}
\begin{aligned}
\frac{w^{n+1}P_{n+4}(w)}{\prod_{k=0}^n(w^2-\gamma_k^2)}
={}&
c_{n+4}w^3+
c_{n+3}w^2+
\Big(c_{n+2}+c_{n+4}\sum_{k=0}^n\gamma_k^2\Big)w+\Big(c_{n+1}+c_{n+3}\sum_{k=0}^n\gamma_k^2\Big)
\\
+
\Big(c_{-n-2}+&c_{-n-4}\sum_{j=0}^n\gamma_j^{-2}\Big)\frac{1}{w\prod_{k=0}^n(-\gamma_k^{2})}
+\frac{c_{-n-3}w^{-2}}{\prod_{k=0}^n(-\gamma_k^{2})}+\frac{c_{-n-4}w^{-3}}{\prod_{k=0}^n(-\gamma_k^{2})}
\\
&+\frac{1}{2}\sum\limits_{k=0}^{n}\frac{\gamma_k^n}{\prod_{j\ne{k}}(\gamma_k^2-\gamma_j^2)}\left(\frac{P_{n+4}(\gamma_k)}{w-\gamma_k}+(-1)^n\frac{P_{n+4}(-\gamma_k)}{w+\gamma_k}\right).
\end{aligned}    
\end{equation}
It is straightforward to verify that the asymptotics of the function on the left hand side as $w\to\infty$ has the form
$$
\frac{w^{n+1}P_{n+4}(w)}{\prod_{k=0}^n(w^2-\gamma_k^2)}=p_3(w)+\Bigg(c_n+c_{n+2}\sum_{k=0}^{n}\gamma_k^2+c_{n+4}\Big(\sum_{k=0}^{n}\gamma_k^2\Big)^2-c_{n+4}\!\!\!\!\sum_{0\le{j}<k\le{n}}(\gamma_j\gamma_k)^2\Bigg)w^{-1}+O(w^{-2}),
$$
where $p_3(w)$ is some third degree polynomial.  Equating the coefficient at $w^{-1}$ to the sum of residues at all finite points we obtain the identity:
\begin{multline}\label{eq:sumres3}
c_n+c_{n+2}\sum_{k=0}^{n}\gamma_k^2+c_{n+4}\Big(\sum_{k=0}^{n}\gamma_k^4+\!\!\!\!\sum_{0\le{j}<k\le{n}}(\gamma_j\gamma_k)^2\Big)
\\
=\Big(c_{-n-2}+c_{-n-4}\sum_{j=0}^n\gamma_j^{-2}\Big)\frac{(-1)^{n+1}}{\prod_{k=0}^n\gamma_k^{2}}
+\frac{1}{2}\sum_{k=0}^{n}\frac{\gamma_k^n(P_{n+4}(\gamma_k)+(-1)^nP_{n+4}(-\gamma_k))}{\prod_{j\ne{k}}(\gamma_k^2-\gamma_j^2)}
\end{multline}
valid for any Laurent polynomial $P_{n+4}(w)$ with terms $w^k$, $k=-n-4,\ldots,n+4$.
\bigskip

\subsection{Application to trigonometric identities}\label{sec:appl.to.trig}
The trigonometric form of formula \eqref{eq:w.basic.expression} is presented in the following corollary (note that the indexing of vectors $\a$ and $\b$ in this section starts with zero in agreement with \cite{Chu}, so that $\sin(z-\b)=\prod_{j=0}^{n}\sin(z-b_j)$).
\begin{corollary}
Define $\mathcal B=\sum_{j=0}^nb_j$ and 
\begin{equation}\label{eq:Ej-defined}
E_{j}^\pm=(\pm 2i)^{n+1}h_{j}(e^{\pm 2ib_0},\dots,e^{\pm 2ib_n})e^{\pm i\mathcal B}
=(\pm 2i)^{n+1}\!\!\sum_{0\le k_1\le \cdots \le k_j\le n}\!\! e^{\pm2i(b_{k_1}+\dots+b_{k_j})\pm i\mathcal B}.
\end{equation}
For each integer $m\ge0$ and integer $\vartheta\le m+n$ having the same parity as $m+n$ we have
\begin{align}\label{eq:exp.positive.powers}
        \frac{e^{imz}}{\sin(z-\b)}
        &=
        \sum\nolimits_{j=0}^{(m-n-\vartheta)/2-1}
        E_{j}^+ e^{i(m-n-2j-1)z}
        +
        \sum_{k=0}^n\frac{e^{i(m-\vartheta)b_k}e^{i\vartheta z}}
        {\sin(b_k-\b_{[k]})\sin(z-b_k)},
    \\
    \label{eq:exp.negative.powers}
        \frac{e^{-imz}}{\sin(z-\b)}
        &=
        \sum\nolimits_{j=0}^{(m-n-\vartheta)/2-1}
        E_{j}^- e^{i(n+2j+1-m)z}
        +
        \sum_{k=0}^n\frac{e^{-i(m-\vartheta)b_k}e^{-i\vartheta z}}
        {\sin(b_k-\b_{[k]})\sin(z-b_k)}.
    \end{align}
\end{corollary}
\begin{proof}
Substituting $w=e^{iz}$ and $\gamma_k=e^{ib_k}$ into  formula \eqref{eq:w.basic.expression} turns it into
\begin{multline*}
    \frac{e^{i(m+n)z}}{\prod_{k=0}^n(e^{2iz}-e^{2ib_k})}
    =
    \sum\nolimits_{j=0}^{(m-n-\vartheta)/2-1}
    h_{j}(e^{2ib_0},\dots,e^{2ib_n}) e^{i(m-n-2-2j)z}
    \\
    +
    \sum_{k=0}^n\frac{e^{i(m+n-\vartheta)b_k} e^{i\vartheta z}}
    {\prod_{\substack{j=0\\j\ne k}}^{n}(e^{2ib_k}-e^{2ib_j})(e^{2iz}-e^{2ib_k})},
\end{multline*}
whence due to~$e^{2iz}-e^{2ib_k}=2i e^{iz+ib_k}\sin(z-b_k)$ we arrive at
\[
    \frac{e^{i(m-1)z-i\mathcal B}}{\sin(z-\b)}
    =
    e^{-i\mathcal B}\sum\nolimits_{j=0}^{(m-n-\vartheta)/2-1}
    E_{j}^+ e^{i(m-n-2-2j)z}
    +
    \sum_{k=0}^n\frac{e^{i(m-\vartheta)b_k+i(\vartheta-1)z-i\mathcal B}}{\sin(b_k-\b_{[k]})\sin(z-b_k)}.
\]
This yields \eqref{eq:exp.positive.powers}. 
Analogously, formula \eqref{eq:exp.negative.powers} follows from~\eqref{eq:w.basic.expression} after substitution~$w=e^{-iz}$ and~$\gamma_k=e^{-ib_k}$.
\end{proof}

Our first application of the above corollary leads immediately to a set of generalizations of \cite[eq. (10)]{Chu}. For every integer~$\vartheta$ satisfying~$m-n\le \vartheta\le m+n$ and of the same parity as~$m+n$, that is for
\[
\vartheta\in\{m-n, m-n+2,\dots, m+n\},
\]
the sums containing~$E_j^\pm$ vanish on the right-hand sides of~\eqref{eq:exp.positive.powers} and~\eqref{eq:exp.negative.powers} leading to 
\begin{equation}\label{eq:ex0.trig.poly}
  \frac{\sin(mz-a)}{\sin(z-\b)}=
    \sum_{k=0}^n\frac{\sin(mb_k-a + \vartheta(z -b_k))}
    {\sin(b_k-\b_{[k]})\sin(z-b_k)}.
\end{equation}
If furthermore,  $m\le n$ choosing $\vartheta\in\{m-n, m-n+2,\dots, n-m\}$, we will also have ~$m-n\le -\vartheta\le n-m$, so we can use~$-\vartheta$ instead of~$\vartheta$ in one of the
formulas~\eqref{eq:exp.positive.powers}--\eqref{eq:exp.negative.powers}, whence (cf.~\cite[eq.~(12)--(13)]{Chu})
\[
    \frac{\sin(mz-a)}{\sin(z-\b)}=\sum_{k=0}^n\frac{\sin(mb_k-a) e^{\pm i \vartheta (z -b_k)}}
    {\sin(b_k-\b_{[k]})\sin(z-b_k)}=\sum_{k=0}^n\frac{\sin(mb_k-a) \cos\big(\vartheta (z-b_k)\big)}
  {\sin(b_k-\b_{[k]})\sin(z-b_k)}
\]
with the second equality obtained by taking the arithmetic average of the ``$+$'' and ``$-$'' cases of the first equality.

Recall that~$E_j^\pm =(\pm 2i)^{n+1}\sum_{0\le k_1\le \cdots \le k_j\le n}
e^{\pm 2i(b_{k_1}+\dots+b_{k_j})\pm i\mathcal B}$, so that for integer~$m>n$ we have
\begin{align*}
    G_{j,m}(z)
    &\overset{\textup{def}}{{}={}}
    (2i)^{-1} \left(E_j^+ e^{-ia}e^{i(m-n-2j-1)z}
    -
    E_j^- e^{ia}e^{-i(m-n-2j-1)z}\right)
    \\
    &=
    (2i)^{n} \!\!\!\!\!\!\sum_{0\le k_1\le \cdots \le k_j\le n}\!\!\!
    \left(
    e^{i(2\sum_{l=1}^jb_{k_l}+\mathcal B-a + (m-n-2j-1)z)}
    +
    (-1)^{n} e^{-i(2\sum_{l=1}^jb_{k_l}+\mathcal B-a + (m-n-2j-1)z)}\right)
    \\
    &= \begin{cases}
        \displaystyle
        (-4)^{n/2}\;2\sum_{\!\!0\le k_1\le \cdots \le k_j\le n\!\!}
        \cos\Big[2\sum_{l=1}^jb_{k_l}+\mathcal B-a + (m-n-2j-1)z\Big],&
        \text{if $n$ is even;}\\[5pt]
        \displaystyle
        (-4)^{\frac{n+1}2}\sum_{\!\!0\le k_1\le \cdots \le k_j\le n\!\!}
        \sin\Big[2\sum_{l=1}^jb_{k_l}+\mathcal B-a + (m-n-2j-1)z\Big],&
        \text{if $n$ is odd.}\\
    \end{cases}
\end{align*}
In particular,
\begin{equation*}
    G_{0,m}(z)
    = \begin{cases}
        \displaystyle
        (-4)^{n/2}\;2
        \cos\big[\mathcal B-a + (m-n-1)z\big],&
        \text{if $n$ is even;}\\[5pt]
        \displaystyle
        (-4)^{\frac{n+1}2}
        \sin\big[\mathcal B-a + (m-n-1)z\big],&
        \text{if $n$ is odd.}\\
    \end{cases}
\end{equation*}
Applying \eqref{eq:exp.positive.powers} and~\eqref{eq:exp.negative.powers} for any integer~$\vartheta\le m+n$ that has the same parity as~$m+n$ we obtain a generalization of~\eqref{eq:ex0.trig.poly} in the form 
\begin{align}
\frac{\sin(mz-a)}{\sin(z-\b)}
{}={}&
\sum\nolimits_{j=0}^{(m-n-\vartheta)/2-1}
E_j^+ e^{-ia}e^{i(m-n-2j-1)z}
-
\sum\nolimits_{j=0}^{(m-n-\vartheta)/2-1}
E_j^- e^{ia}e^{-i(m-n-2j-1)z}
\nonumber
\\&+
\sum_{k=0}^n\frac{e^{i(mb_k-a)+i\vartheta(z-b_k)}-e^{-i(mb_k-a)-i\vartheta(z-b_k)}}
{\sin(b_k-\b_{[k]})\sin(z-b_k)}
\nonumber
  \\[2pt]
\label{eq:ex1.trig.poly}  
{}={}&
\sum\nolimits_{j=0}^{(m-n-\vartheta)/2-1} G_{j,m}(z)
+
\sum_{k=0}^n\frac{\sin\big(mb_k-a+\vartheta(z-b_k)\big)}
       {\sin(b_k-\b_{[k]})\sin(z-b_k)}.
\end{align}

For instance, the last formula in the case~$m= n+1$ with~$\vartheta=1$ and~$\vartheta=-1$, respectively, yields
\begin{align*}
\frac{\sin\big((n+1)z-a\big)}{\sin(z-\b)}
&=
\sum_{k=0}^n\frac{\sin\big(nb_k-a+z\big)}
{\sin(b_k-\b_{[k]})\sin(z-b_k)}
\\
&=
\sum_{k=0}^n\frac{\sin\big((n+2)b_k-a-z\big)}
{\sin(b_k-\b_{[k]})\sin(z-b_k)}
+\begin{cases}
    (-4)^{n/2}\cdot 2\cos(\mathcal B-a),&\text{if $n$ is even;}\\
    (-4)^{(n+1)/2}\cdot\sin(\mathcal B-a),&\text{if $n$ is odd.}
\end{cases}
\end{align*}
Moreover, combinations of~\eqref{eq:exp.negative.powers} and~\eqref{eq:exp.positive.powers} with
choices of~$\vartheta\in\{-1,1\}$ other than in~\eqref{eq:ex1.trig.poly} can also be of interest; for~$m=n+1$ they additionally give
\begin{align*}
\frac{\sin\big((n+1)z-a\big)}{\sin(z-\b)}
&=
(-2i)^{n}e^{ia-i\mathcal B}
+
\sum_{k=0}^n\frac{\sin\big((n+1)b_k-a\big) e^{i(z-b_k)}}
                                                                    {\sin(b_k-\b_{[k]})\sin(z-b_k)}
  \\
&=
(2i)^{n}e^{i\mathcal B-ia}
+
\sum_{k=0}^n\frac{\sin\big((n+1)b_k-a\big) e^{i(b_k-z)}}
{\sin(b_k-\b_{[k]})\sin(z-b_k)}
\\
&\hspace{-5em}=
\sum_{k=0}^n\frac{\sin\big((n+1)b_k-a\big) \cot(z-b_k)}
{\sin(b_k-\b_{[k]})}
+
\begin{cases}
(-4)^{n/2}\cdot\cos(\mathcal B-a),&\text{if $n$ is  even;}\\
-(-4)^{\frac{n-1}2}\cdot 2\sin(\mathcal B-a),&\text{if $n$ is odd.}
\end{cases}
\end{align*}

Other curious particular cases of~\eqref{eq:ex1.trig.poly} include
\begin{align*}
\frac{\sin\big((n+2)z-a\big)}{\sin(z-\b)}
&=
\sum_{k=0}^n\frac{\sin\big((n+2)b_k-a\big)}
{\sin(b_k-\b_{[k]})\sin(z-b_k)}
+\begin{cases}
(-4)^{\frac{n}2}\cdot 2\cos(z+\mathcal B-a),&\text{if $n$ is even;}\\
(-4)^{\frac{n-1}2}\cdot\sin(z+\mathcal B-a),&\text{if $n$ is odd,}
\end{cases}
\end{align*}
and
\begin{align*}
\frac{\sin\big((n+3)z-a\big)}{\sin(z-\b)}
&=
\sum_{k=0}^n\frac{\sin\big((n+2)b_k-a+z\big)}
{\sin(b_k-\b_{[k]})\sin(z-b_k)}
+\begin{cases}
(-4)^{\frac{n}2}\cdot 2\cos(2z+\mathcal B-a),&\text{if $n$ is even;}\\
(-4)^{\frac{n-1}2}\cdot\sin(2z+\mathcal B-a),&\text{if $n$ is odd.}
\end{cases}
\end{align*}
It is clear that replacing~$a$ with~$a+\frac \pi2$ and/or~$b_k$
with~$b_k+\frac \pi2$ one immediately obtains counterparts of the above formulas, in which the corresponding sines on the left hand side are replaced by cosines.
\medskip

We can further substitute particular forms of the polynomial $P(w)$ in \eqref{eq:Chuplus1}  and \eqref{eq:Chuplus2}. For instance, following the structure of Braaksma's and Meijer's identities \eqref{eq:mainBraaksma} and \eqref{eq:MeijerLemma4} and examples from \cite{Chu} define
\begin{equation}\label{eq:Psineprod}
P(e^{iz})=\prod_{k=0}^{n+1}\sin(z-a_k)\eqqcolon \sin(z-\a),
\end{equation}
so that (cf. Chu \cite[p.233]{Chu})
$$
P(w)=\frac{e^{-i\mathcal{A}}}{(2i)^{n+2}w^{n+2}}\prod_{k=0}^{n+1}(w^2-e^{2ia_k})
=\frac{e^{-i\mathcal{A}}}{(2i)^{n+2}}w^{n+2}+\delta w^{n}+\cdots+(-1)^{n+2}e^{i\mathcal{A}}\frac{w^{-n-2}}{(2i)^{n+2}},
$$
where $\mathcal{A}=\sum_{k=0}^{n+1}a_k$ and $(-1)^nP(-w)=P(w)$.  Substituting this into \eqref{eq:Chuplus1} we recover formula \eqref{eq:BraaksmaKappa1}:
\begin{equation}\label{eq:ChuKappa1}
\frac{\sin(z-\a)}{\sin(z-\b)}=\sin(z-\nu)
+\sum_{k=0}^{n}\frac{\sin(b_k-\a)}{\sin(z-b_k)\sin(b_k-\b_{[k]})},
\end{equation}
where
$$
\nu=\mathcal{A}-\mathcal{B}=\sum_{k=0}^{n+1}a_k-\sum_{k=0}^{n}b_k.
$$
Using $P(w)$ from \eqref{eq:Psineprod} in \eqref{eq:sumres1}, we obtain the following exotic identity
$$
4\sum_{k=0}^{n}\frac{e^{i(\nu+b_k)}
    \sin(b_k-\a)}{\sin(b_k-\b_{[k]})}=\sum_{k=0}^{n+1}e^{2ia_k}-\sum_{k=0}^{n}e^{2ib_k}-e^{2i\nu}.
$$
Separating real and imaginary parts of the above identity we can get similar formulas with cosines and sines in place of  exponentials. 

Setting~$a=a_{n+1}$ and $a_k=b_k-\pi/2$ for $k=0,\ldots,n$ in \eqref{eq:ChuKappa1} we obtain:
$$
\sin(z-a)\cot(z-\b)=\sin(z-a+(n+1)\pi/2)
+\sum_{k=0}^{n}\frac{\sin(b_k-a)}{\sin(z-b_k)}\cot(b_k-\b_{[k]}).
$$

Suppose 
\begin{equation}\label{eq:Psineprod2}
P(e^{iz})=\prod_{k=0}^{n+2}\sin(z-a_k) \eqqcolon \sin(z-\a),
\end{equation}
so that (cf. Chu \cite[p.233]{Chu})
$$
P(w)\!=\!\frac{e^{-i\mathcal{A}}}{(2i)^{n+3}w^{n+3}}\prod_{k=0}^{n+2}(w^2-e^{2ia_k})
\!=\!\frac{e^{-i\mathcal{A}}}{(2i)^{n+3}}w^{n+3}-\frac{e^{-i\mathcal{A}}}{(2i)^{n+3}}w^{n+1}\sum_{k=0}^{n+2}e^{2ia_k}+\cdots+(-1)^{n+3}e^{i\mathcal{A}}\frac{w^{-n-3}}{(2i)^{n+3}},
$$
where $\mathcal{A}=\sum_{k=0}^{n+2}a_k$ and $(-1)^nP(-w)=-P(w)$.  Substituting this into \eqref{eq:Chuplus2} we obtain:
\begin{multline*}
(2i)^{-n-1}e^{-i\mathcal{B}}\frac{\sin(z-\a)}{\sin(z-\b)}=
\frac{e^{-i\mathcal{A}}}{(2i)^{n+3}}w^2+\frac{(-1)^{n+1}}{w^2\prod_{k=0}^{n}\gamma_k^2}
\frac{(-1)^{n+3}e^{i\mathcal{A}}}{(2i)^{n+3}}
\\
+\frac{e^{-i\mathcal{A}}}{(2i)^{n+3}}\sum_{k=0}^{n}e^{2ib_k}-\frac{e^{-i\mathcal{A}}}{(2i)^{n+3}}\sum_{k=0}^{n+2}e^{2ia_k}+\sum\limits_{k=0}^{n}\frac{\gamma_k^{n+1}P(\gamma_k)}{\prod_{j\ne{k}}(\gamma_k^2-\gamma_j^2)(w^2-\gamma_k^2)},
\end{multline*}
or
\begin{multline}\label{eq:ChuKappa2}
\frac{\sin(z-\a)}{\sin(z-\b)}=
-\frac{1}{2}\cos(\mathcal{B}-\mathcal{A}+2z)
+\frac{e^{i(\mathcal{B}-\mathcal{A})}}{4}\Big(\sum_{k=0}^{n+2}e^{2ia_k}-\sum_{k=0}^{n}e^{2ib_k}\Big)
\\
+\sum\limits_{k=0}^{n}\frac{e^{i(b_k-z)}\sin(b_k-\a)}{\sin(z-b_k)\sin(b_k-\b_{[k]})}.
\end{multline}
Separating the real and imaginary parts leads to the identities:
\begin{multline}\label{eq:ChuKappa2Re}
\frac{\sin(z-\a)}{\sin(z-\b)}=
-\frac{1}{2}\cos(\mathcal{B}-\mathcal{A}+2z)
+\frac{1}{4}\sum_{k=0}^{n+2}\cos(\mathcal{B}-\mathcal{A}+2a_k)
-\frac{1}{4}\sum_{k=0}^{n}\cos(\mathcal{B}-\mathcal{A}+2b_k)
\\
+\sum\limits_{k=0}^{n}\frac{\sin(b_k-\a)}{\sin(b_k-\b_{[k]})}\cot(z-b_k)
\end{multline}
and
\begin{equation}\label{eq:ChuKappa2Im}
4\sum\limits_{k=0}^{n}\frac{\sin(b_k-\a)}{\sin(b_k-\b_{[k]})}=
\sum_{k=0}^{n+2}\sin(\mathcal{B}-\mathcal{A}+2a_k)-\sum_{k=0}^{n}\sin(\mathcal{B}-\mathcal{A}+2b_k).
\end{equation}
Formula \eqref{eq:sumres1} reduces to triviality in this case.

In a similar fashion, taking equation \eqref{eq:Chuplus3}
with~$\sin(z-\a)\coloneqq \prod_{k=0}^{n+3}\sin(z-a_k)$ we recover formula~\eqref{eq:Fkappa3}:
\begin{multline}\label{eq:ChuKappa3}
\frac{\sin(z-\a)}{\sin(z-\b)}=
-\frac{1}{4}\sin(3z+\mathcal{B}-\mathcal{A})
+\frac{1}{4}\sum_{k=0}^{n+3}\sin(z+\mathcal{B}-\mathcal{A}+2a_k)
\\
-\frac{1}{4}\sum_{k=0}^{n}\sin(z+\mathcal{B}-\mathcal{A}+2b_k)
+\sum_{k=0}^{n}\frac{\sin(b_k-\a)}{\sin(z-b_k)\sin(b_k-\b_{[k]})},
\end{multline}
while application of formula \eqref{eq:sumres3} yields the following identity
\begin{multline*}
\sum_{0\le j_1<j_2\le n+3}e^{2i(a_{j_1}+a_{j_2})}
-\sum_{0\le j\le n+3}e^{2ia_{j}}\sum_{0\le k\le n}e^{2ib_{k}}
+\sum_{0\le k\le n}e^{4ib_{k}}
+\sum_{0\le k_1<k_2\le n}e^{2i(b_{k_1}+b_{k_2})}
\\
=e^{2i(\mathcal{A}-\mathcal{B})}\sum_{0\le j\le n+3}e^{-2ia_{j}}-e^{2i(\mathcal{A}-\mathcal{B})}\sum_{0\le k\le n}e^{-2ib_{k}}
+16e^{i(\mathcal{A}-\mathcal{B})}\sum_{k=0}^{n}\frac{e^{ib_k}\sin(b_k-\a)}{\sin(b_k-\b_{[k]})}.
\end{multline*}
Separating real and imaginary parts of the above identity we can get similar formulas with cosines and sines in place of exponentials.

We can also use a generic trigonometric polynomial in~\eqref{eq:Chuplus1} and~\eqref{eq:Chuplus2} by writing~$w=e^{iz}$ and~$\gamma_k=e^{ib_k}$. For instance, taking $T(z)=P_{n+3}(e^{iz})$ after some calculations brings \eqref{eq:Chuplus2} to the form
\begin{align}\nonumber
&\frac{(2i)^{-n-1}T(z)}{\sin(z-\b)}
=c_{n+3}e^{2iz+i\mathcal B}+(-1)^{n+1}c_{-n-3}e^{-2iz-i\mathcal B}+
e^{i\mathcal B}\Big(c_{n+1}+c_{n+3}\sum\nolimits_{k=0}^n e^{2ib_j}\Big)
\\\label{eq:TigFormExample}
&+c_{n+2}e^{iz+i\mathcal B}+(-1)^{n+1} c_{-n-2}e^{-iz-i\mathcal B}
+\sum_{k=0}^n\frac{ {T(b_k)\cos\frac{z-b_k}2} +(-1)^n i{T(b_k+\pi)\sin\frac{z-b_k}2}}
{(2i)^{n+1}e^{i\frac{z-b_k}2}\sin(z-b_k)\sin(b_k-\b_{[k]})}.
\end{align}
\medskip

Our next goal is to extend \eqref{eq:TigFormExample} to a $2\pi$-periodic trigonometric polynomial~$T(z)=\sum_{t=-m}^m c_t e^{itz}$ of arbitrary degree $m$. One can split such a polynomial into its $\pi$-periodic and $\pi$-antiperiodic parts according to, respectively:
\begin{align*}
  T_p(z)&=\frac{T(z)+T(z+\pi)}2=\sum_{-m\le t\le m}^{t~\text{even}} c_{t} e^{itz}=\sum_{l=-\lfloor m/2\rfloor}^{\lfloor m/2\rfloor} c_{2l} e^{2ilz}
          \qquad\text{and\hspace{-3em}}
  \\
  T_a(z)&=\frac{T(z)-T(z+\pi)}2=\sum_{-m\le t\le m}^{t~\text{odd}} c_{t} e^{itz} =\!\!\sum_{l=-\lfloor (m+1)/2\rfloor}^{\lfloor (m-1)/2\rfloor}\!\! c_{2l+1} e^{i(2l+1)z}
          ,
\end{align*}
so that~$T(z)=T_p(z)+T_a(z)$. In the sequel we use the right continuous version of the sign function, i.e. 
$$
\sign(t)=\begin{cases}
 1, & t\ge0
 \\
 -1, & t<0
\end{cases}.
$$

\begin{theorem}\label{th:new}
Recall that $\mathcal B=\sum_{k=0}^{n}b_k$ and define
\begin{equation}\label{eq:FE-defined}
    F_k^\pm=\sum\nolimits_{j=0}^{\lfloor\frac{m-n-k-1}{2}\rfloor} c_{\pm (2j+k+n+1)}E_j^\pm, 
\quad\text{where}\quad
    E_j^\pm=(\pm 2i)^{n+1}\!\!\!\!\sum_{0\le k_1\le \cdots \le k_j\le n}\!\!\!\! e^{\pm2i(b_{k_1}+\dots+b_{k_j})\pm i\mathcal B}  
\end{equation}
retains its meaning from \eqref{eq:Ej-defined}. Empty sums here are understood as zero \emph{(}in particular, $F_0^\pm=0$ unless~$m>n$\emph{)}. For any exponential polynomial~$T(z)=\sum_{t=-m}^m c_t e^{itz}$ the following identities hold\emph{:}
\begin{subequations}
\label{eq:gen.trig.poly}
\begin{align}
\frac{T(z)}{\sin(z-\b)}
\label{eq:gen.trig.poly.m}
&=
\sum\nolimits_{k=1}^{m-n-1}
\left(F_k^+ e^{ikz} + F_k^- e^{-ikz}\right)
+
\sum_{k=0}^n\frac{\sum_{t=-m}^m 
    c_t e^{itb_k+i\nu_{t,n}(z-b_k)} }
{\sin(b_k-\b_{[k]})\sin(z-b_k)}
\\
\label{eq:gen.trig.poly.p}
&=
\sum\nolimits_{k=0}^{m-n-1}
\left( F_k^+ e^{ikz} + F_k^- e^{-ikz}\right)
+
\sum_{k=0}^n\frac{\sum_{t=-m}^m 
    c_t e^{itb_k-i\nu_{t,n}(z-b_k)} }
{\sin(b_k-\b_{[k]})\sin(z-b_k)}
,
\end{align}
\end{subequations}
where $\nu_{t,n}=\sign(t)\cdot\big(t+n-2\lfloor(t+n)/2\rfloor\big)$.  Furthermore, for even~$n$ we have
\begin{align}
\nonumber
\frac{T(z)}{\sin(z-\b)}
&=
F_0^\mp + \sum\nolimits_{k=1}^{m-n-1}
\left( F_k^+ e^{ikz} + F_k^- e^{-ikz}\right)
+
\sum_{k=0}^n\frac{T_p(b_k)+T_a(b_k) e^{\pm i (z-b_k)} }
{\sin(b_k-\b_{[k]})\sin(z-b_k)}
\\\label{eq:gen.trig.poly.pm.e}
=\frac{F_0^+ + F_0^-}2 + 
&\sum\nolimits_{k=1}^{m-n-1}
\left( F_k^+ e^{ikz} + F_k^- e^{-ikz}\right)
+\sum_{k=0}^n\frac{T_p(b_k)+T_a(b_k) \cos(z-b_k) }
{\sin(b_k-\b_{[k]})\sin(z-b_k)}
,
\end{align}
and for odd~$n$ we have
\begin{align}
\nonumber
\frac{T(z)}{\sin(z-\b)}
&=F_0^\mp + \sum\nolimits_{k=1}^{m-n-1}
\left( F_k^+ e^{ikz} + F_k^- e^{-ikz}\right)
+
\sum_{k=0}^n\frac{T_p(b_k) e^{\pm i (z-b_k)} + T_a(b_k) }
{\sin(b_k-\b_{[k]})\sin(z-b_k)}
\\\label{eq:gen.trig.poly.pm.o}
=\frac{F_0^+ + F_0^-}2 + &\sum\nolimits_{k=1}^{m-n-1}
\left( F_k^+ e^{ikz} + F_k^- e^{-ikz}\right)
+
\sum_{k=0}^n\frac{T_p(b_k)\cos(z-b_k) + T_a(b_k) }
{\sin(b_k-\b_{[k]})\sin(z-b_k)}.
\end{align}
Note that $F_0^{\pm}=0$ if $m-n<1$.
\end{theorem}

\begin{proof}[Proof of Theorem~\ref{th:new}]
    Denote
    \[
        \mathcal E_j^\pm=\begin{cases}
            E_{j/2}^\pm, &\text{if~$j$ is even};\\
            0, &\text{if~$j$ is odd.}
        \end{cases}
    \]
    Then  we have
    \begin{multline*}
        \sum_{t=0}^mc_t\sum_{j=0}^{\!\!\lfloor(t-n)/2\rfloor-1\!\!}
        E_{j}^+ e^{i(t-n-2j-1)z}
        =
        \sum_{t=0}^{\! m-n-2 \!} c_{t+n+2}\sum_{j=0}^{\lfloor t/2 \rfloor}
        E_{j}^+ e^{i(t-2j+1)z}
        \\
        =
        \!\!\sum_{t=0}^{m-n-2}\! c_{t+n+2}\sum_{k=0}^{t}
        \mathcal E_{k}^+ e^{i(t+1-k)z}
        =
        \!\!\sum_{t=0}^{m-n-2}\;\sum_{k=1}^{t+1}
         c_{t+n+2}\, \mathcal E_{t+1-k}^+ e^{ikz}
        =
        \!\!\sum_{k=1}^{m-n-1}\! e^{ikz}\sum_{ t=k-1 }^{m-n-2}\!
        c_{t+n+2}\, \mathcal E_{t+1-k}^+
        \\
        =
        \!\!\sum_{k=1}^{m-n-1}\! e^{ikz}\sum_{t=0 }^{m-n-k-1}\!\!
        c_{t+k+n+1}\, \mathcal E_{t}^+
        =
        \!\!\sum_{k=1}^{m-n-1}\! e^{ikz}\sum_{ j=0 }^{\lfloor(m-n-1-k)/2\rfloor}\!\!\!
        c_{2j+k+n+1} E_{j}^+
        ,
    \end{multline*}
    where the first equality results from the fact that the  summands on the left-hand side with~$t\le n+1$ do not contribute the sum. In a similar fashion,
    \begin{equation*}
        \sum_{t=1}^mc_{-t}\!\!\sum_{j=0}^{\!\!\lfloor(t-n)/2\rfloor-1\!\!}
        E_{j}^- e^{-i(t-n-2j-1)z}
       =\sum_{k=1}^{m-n-1}\! e^{-ikz}\sum_{ j=0 }^{\lfloor(m-n-1-k)/2\rfloor}\!\!\!
        c_{-2j-k-n-1} E_{j}^-.
    \end{equation*}
We now apply \eqref{eq:exp.positive.powers} (for $t\ge0$) and \eqref{eq:exp.negative.powers} (for $t<0$) to each summand on the left-hand side of \eqref{eq:gen.trig.poly.m} of the form 
$$
\frac{c_te^{itz}}{\sin(z-\b)}
$$
while choosing $\vartheta=\varepsilon_{t,n}\in\{0,1\}$ for each $t$ from the condition $(t-n-\varepsilon_{t,n})/2=\lfloor(t-n)/2\rfloor$. We thus obtain the equality
    \begin{multline*}
        \frac{\sum_{t=-m}^m c_te^{itz}}{\sin(z-\b)}
        =
        \!\sum_{k=1}^{m-n-1}\left( e^{ikz}\sum_{ j=0 }^{\lfloor(m-n-1-k)/2\rfloor}\!\!\!
        c_{2j+k+n+1} E_{j}^+
        +
        e^{-ikz}\sum_{ j=0 }^{\lfloor(m-n-1-k)/2\rfloor}\!\!\!
        c_{-2j-k-n-1} E_{j}^-\right)
        \\
        +\sum_{k=0}^n\frac{\sum_{t=0}^m c_t e^{itb_k}e^{i\varepsilon_{t,n}(z-b_k)}+\sum_{t=-1}^{-m} c_{t} e^{itb_k}e^{-i\varepsilon_{t,n}(z-b_{k})}}
        {\sin(b_k-\b_{[k]})\sin(z-b_k)}
        ,
    \end{multline*}
    which implies~\eqref{eq:gen.trig.poly.m} in view of $\nu_{t,n}=\varepsilon_{t,n}$ for $t\ge0$ and $\nu_{t,n}=-\varepsilon_{t,n}$ for $t<0$.

A similar calculation as above, but choosing rounding upwards shows that
\begin{equation*}
    \sum_{t=1}^mc_{\pm t}\!\!\sum_{j=0}^{\!\!\lfloor(t-n+1)/2\rfloor-1\!\!}
    E_{j}^\pm e^{\pm i(t-n-2j-1)z}=
    \sum_{k=0}^{m-n-1}\! e^{\pm ikz}\sum_{ j=0 }^{\lfloor(m-n-1-k)/2\rfloor}\!\!\!
    c_{\pm(2j+k+n+1)} E_{j}^\pm.
\end{equation*}
Therefore, application of~\eqref{eq:exp.positive.powers} (for $t\ge0$)
and~\eqref{eq:exp.negative.powers} (for $t<0$) to each summand on the left-hand side
of~\eqref{eq:gen.trig.poly.p} while choosing $\vartheta=-\varepsilon_{t,n}\in\{0,-1\}$
with~$\varepsilon_{t,n}$ as above yields:
\begin{multline*}
    \frac{\sum_{t=-m}^m c_te^{itz}}{\sin(z-\b)}
    =
    \!\sum_{k=0}^{m-n-1}\left( e^{ikz}\sum_{ j=0 }^{\lfloor(m-n-1-k)/2\rfloor}\!\!\!
        c_{2j+k+n+1} E_{j}^+
        +
        e^{-ikz}\sum_{ j=0 }^{\lfloor(m-n-1-k)/2\rfloor}\!\!\!
        c_{-2j-k-n-1} E_{j}^-\right)
    \\
    +\sum_{k=0}^n\frac{\sum_{t=0}^m c_t e^{itb_k}e^{-i\varepsilon_{t,n}(z-b_k)}+\sum_{t=-1}^{-m} c_{t} e^{itb_k}e^{i\varepsilon_{t,n}(z-b_{k})}}
    {\sin(b_k-\b_{[k]})\sin(z-b_k)}
    ,
\end{multline*}
and hence~\eqref{eq:gen.trig.poly.p}.

To establish \eqref{eq:gen.trig.poly.pm.e} and \eqref{eq:gen.trig.poly.pm.o} we apply
\eqref{eq:exp.positive.powers} (for $t\ge0$) and \eqref{eq:exp.negative.powers} (for $t<0$) with
$\vartheta=\nu_{t,n}$ to each summand on the left-hand side of \eqref{eq:gen.trig.poly.pm.e} to
get
    \begin{multline}\label{eq:gen.trig.poly.pm.calc1}
        \frac{\sum_{t=-m}^m c_te^{itz}}{\sin(z-\b)}
        =\!\sum_{k=1}^{m-n-1}\left(e^{ikz}\!\!\!\!\sum_{ j=0 }^{\lfloor(m-n-1-k)/2\rfloor}\!\!\!
        c_{2j+k+n+1} E_{j}^+
        +
        e^{-ikz}\!\!\!\!\sum_{ j=0 }^{\lfloor(m-n-1-k)/2\rfloor}\!\!\!
        c_{-2j-k-n-1} E_{j}^-\right)
        \\
        +\sum_{j=0}^{\lfloor(m-n-1)/2\rfloor}\!\!\!
        c_{-2j-n-1} E_{j}^{-}
        +\sum_{k=0}^n\frac{S_{n}(k,m)}{\sin(b_k-\b_{[k]})\sin(z-b_k)},
    \end{multline}
    where
$$
S_{n}(k,m)=\begin{cases}
    \displaystyle{\sum_{-m\le t \le m}^{t~\text{even}}}c_{t}e^{itb_k} + 
        e^{i(z-b_k)}\displaystyle{\sum_{-m\le t \le m}^{t~\text{odd}}}c_{t}e^{itb_k}=T_p(b_k)+e^{i(z-b_k)}T_a(b_k), & n~\text{is even;}
        \\[20pt]
        \displaystyle{\sum_{-m\le t \le m}^{t~\text{odd}}}c_{t}e^{itb_k} + 
        e^{i(z-b_k)}\displaystyle{\sum_{-m\le t \le m}^{t~\text{even}}}c_{t}e^{itb_k}=T_a(b_k)+e^{i(z-b_k)}T_p(b_k), & n~\text{is odd}.
\end{cases}
$$
In a similar fashion, on applying~\eqref{eq:exp.positive.powers} for $t\ge0$
and~\eqref{eq:exp.negative.powers} for $t<0$ while~$\vartheta=-\nu_{t,n}$ we arrive at
   \begin{multline}\label{eq:gen.trig.poly.pm.calc2}
        \frac{\sum_{t=-m}^m c_te^{itz}}{\sin(z-\b)}
        =\!\sum_{k=1}^{m-n-1}\left(e^{ikz}\!\!\!\!\sum_{j=0}^{\lfloor(m-n-1-k)/2\rfloor}\!\!\!
        c_{2j+k+n+1}E_{j}^+
        +
        e^{-ikz}\!\!\!\!\sum_{j=0}^{\lfloor(m-n-1-k)/2\rfloor}\!\!\!
        c_{-2j-k-n-1} E_{j}^-\right)
        \\
        +\sum_{j=0}^{\lfloor(m-n-1)/2\rfloor}\!\!\!
        c_{2j+n+1} E_{j}^{+}
        +\sum_{k=0}^n\frac{\hat{S}_{n}(k,m)}{\sin(b_k-\b_{[k]})\sin(z-b_k)},
    \end{multline}
    where
$$
\hat{S}_{n}(k,m)=\begin{cases}
    \displaystyle{\sum_{-m\le t \le m}^{t~\text{even}}}c_{t}e^{itb_k} + 
        e^{-i(z-b_k)}\displaystyle{\sum_{-m\le t \le m}^{t~\text{odd}}}c_{t}e^{itb_k}=T_p(b_k)+e^{-i(z-b_k)}T_a(b_k), & n~\text{is even;}
        \\[20pt]
        \displaystyle{\sum_{-m\le t \le m}^{t~\text{odd}}}c_{t}e^{itb_k} + 
        e^{-i(z-b_k)}\displaystyle{\sum_{-m\le t \le m}^{t~\text{even}}}c_{t}e^{itb_k}=T_a(b_k)+e^{-i(z-b_k)}T_p(b_k), & n~\text{is odd}.
\end{cases}
$$
Depending on parity of~$n$, we combine the respective cases of the
identities~\eqref{eq:gen.trig.poly.pm.calc1} and~\eqref{eq:gen.trig.poly.pm.calc2} to obtain the
first equalities in~\eqref{eq:gen.trig.poly.pm.e} for even~$n$ and~\eqref{eq:gen.trig.poly.pm.o}
for odd~$n$. The second equalities follow by taking arithmetic averages
of~\eqref{eq:gen.trig.poly.pm.calc1} and~\eqref{eq:gen.trig.poly.pm.calc2} instead of combining
them.
\end{proof}

Theorem~\ref{th:new} enables us to extend formulas \eqref{eq:ChuKappa1}, \eqref{eq:ChuKappa2}, \eqref{eq:ChuKappa2Re}, \eqref{eq:ChuKappa2Im} and  \eqref{eq:ChuKappa3} to the case of arbitrary number of sine functions in the numerator. More specifically, define 
\begin{equation}\label{eq:T_prod_sin}
    T(z)=\prod_{t=0}^r\sin(z-a_t)\eqqcolon \sin(z-\a).
\end{equation}
This trigonometric  polynomial is ~$\pi$-periodic for odd~$r$ and $\pi$-antiperiodic for even~$r$. Conversely, up to multiplication by a constant, every~$\pi$-periodic or~$\pi$-antiperiodic trigonometric polynomial may, in principle, be written in the form~\eqref{eq:T_prod_sin}, which is clear from~\eqref{eq:sin.to.exp} below.

\bigskip
\begin{corollary}\label{cr:sine.over.sine}
Let~$T(z)$ be as in~\eqref{eq:T_prod_sin}. Denote~$\mathcal A=\sum\nolimits_{k=0}^ra_k$, as well as~$\varkappa=r-n$ and~$\lambda=\lfloor\varkappa/2\rfloor$. If~$\varkappa$ is odd, then
\[
\frac{\sin(z-\a)}
{\sin(z-\b)}
=
\sum_{t=0}^{\lambda}
\mathcal F_{\varkappa-2t}(z)
+
\sum_{k=0}^n
\frac{\sin(b_k-\a)}
{\sin(b_k-\b_{[k]})\cdot \sin(z-b_k)}
,
\]
where the first sum is non-void only for~$\varkappa> 0$, and its terms are (with~$F_{\varkappa-2t}^\pm$ keeping their meaning from Theorem~\ref{th:new})
\begin{multline*}
\mathcal F_{\varkappa-2t}(z)
=F_{\varkappa-2t}^+ e^{i(\varkappa-2t)z} + F_{\varkappa-2t}^- e^{-i(\varkappa-2t)z}
\\[2pt]
=
\sum_{m=0}^t \frac{(-1)^{\lambda + t-m}}{2^{\varkappa-1}} \!\!
\sum_{\substack{0\le j_1<\cdots< j_{t-m}\le r\\[2pt] 0\le k_1\le\cdots\le k_m\le n}}
\!\!\!\!\!\!
\sin\Big((\varkappa-2t)z
+ \mathcal B-\mathcal A+2(b_{k_1}+\cdots+b_{k_m} + a_{j_1}+\cdots+a_{j_{t-m}})\Big).
\end{multline*}
For $m=0$ \emph{(}$m=t$\emph{)} the summation in the inner sum is over $j_1,\ldots,j_{t}$ \emph{(}$k_1,\ldots,k_m$\emph{)} and the numbers $b_{k_l}$ \emph{(}$a_{j_l}$\emph{)} are missing in the rightmost parentheses in the argument of the sine function.   

Similarly, if~$\varkappa$ is even, then 
\begin{align} \label{eq:sine.over.sine-odd.kappa}
\frac{\sin(z-\a)}{\sin(z-\b)}
&=
\mathcal F_0^\pm+\sum_{t=1}^{\lambda} \mathcal F_{\varkappa-2t}(z)
+
\sum_{k=0}^n
\frac{\sin(b_k-\a)}
{\sin(b_k-\b_{[k]})}
\cdot\frac{e^{\pm i(z-b_k)}}{ \sin(z-b_k)}
\\ \nonumber
&=
\sum_{t=0}^{\lambda}
\mathcal F_{\varkappa-2t}(z)
+
\sum_{k=0}^n
\frac{\sin(b_k-\a)}
{\sin(b_k-\b_{[k]})}\cdot\cot(z-b_k)
,
\end{align}
where \emph{(}note that $\mathcal F_0^\pm=\mathcal F_{0}(z)=0$ if $\lambda<0$\emph{)}
\begin{align*}
\mathcal F_0^\pm
&=F_0^\mp
=\frac{e^{\mp i\left(\mathcal B-\mathcal A\right)}}{2^{\varkappa}}
\sum_{m=0}^{\lambda} (-1)^{m}\!\!\!
\sum_{\substack{0\le j_1<\cdots< j_{\lambda-m}\le r\\[2pt] 0\le k_1\le\cdots\le k_m\le n}}\!\!\!
e^{\mp 2i\left(b_{k_1}+\cdots+b_{k_m} + a_{j_1}+\cdots+a_{j_{\lambda-m}}\right)},
\\
\mathcal F_{0}(z)&\equiv \tfrac 12 \big({\mathcal F_{0}^+ + \mathcal F_{0}^-}\big)
\\&=
\sum_{m=0}^{\lambda} \cfrac{(-1)^{m \!\!}}{2^\varkappa}\!\!\!
\sum_{\substack{0\le j_1<\cdots< j_{\lambda-m}\le r\\[2pt] 0\le k_1\le\cdots\le k_m\le n}}
\cos\big(\mathcal B-\mathcal A +2(b_{k_1}+\cdots+b_{k_m} + a_{j_1}+\cdots+a_{j_{\lambda-m}})\big),
\end{align*}
and, for~$t=0,1,\dots,\lambda-1$,
\begin{multline*}
\mathcal F_{\varkappa-2t}(z)
= F_{\varkappa-2t}^+ e^{i(\varkappa-2t)z} + F_{\varkappa-2t}^- e^{-i(\varkappa-2t)z}
\\[2pt]
=
\sum_{m=0}^t \frac{(-1)^{\lambda+t-m}}{2^{\varkappa-1}} \!\!
\sum_{\substack{0\le j_1<\cdots< j_{t-m}\le r\\[2pt] 0\le k_1\le\cdots\le k_m\le n}}
\!\!\!\!\!\!
\cos\Big((\varkappa-2t)z
+ \mathcal B-\mathcal A+2(b_{k_1}+\cdots+b_{k_m} + a_{j_1}+\cdots+a_{j_{t-m}})\Big)
.
\end{multline*}

Moreover, for even~$\varkappa=2\lambda$ we additionally have
$$
\sum_{m=0}^{\lambda} \cfrac{(-1)^{m \!\!}}{4^\lambda}\!\!\!
\sum_{\substack{0\le j_1<\cdots< j_{\lambda-m}\le n+2\lambda\\[2pt] 0\le k_1\le\cdots\le k_m\le n}}
\!\!\!\!\!\!
\sin\big(\mathcal B-\mathcal A +2(b_{k_1}+\cdots+b_{k_m} + a_{j_1}+\cdots+a_{j_{t-m}})\big)=\sum_{k=0}^{n}\frac{\prod_{t=0}^{n+2\lambda} \sin(b_k-a_t)}
{\sin(b_k-\b_{[k]})}.
$$
\end{corollary}
\begin{proof}
On the one hand,
\begin{equation}\label{eq:sin.to.exp}
T(z)=(2i)^{-r-1}e^{-i(r+1)z-i\mathcal A} \prod\nolimits_{t=0}^r(e^{2iz}-e^{2ia_t})
= \sum\nolimits_{k=-r-1}^{r+1} c_{k}e^{ikz},
\end{equation}
where~$c_{r+1-k}=0$ if~$k$ is odd, and
\[
c_{r+1-k}=c_{r+1-2t}=(2i)^{-r-1}(-1)^{t}\sum_{0\le j_1<\cdots< j_{t}\le r}
e^{2i(a_{j_1}+\cdots+a_{j_{t}})-i\mathcal A}
\]
if $k\eqqcolon 2t$ is even (it $t=0$, then $c_{r+1}=(2i)^{-r-1}e^{-i\mathcal A}$). However, one can alternatively write
\(
T(z)=(-2i)^{-r-1}e^{i(r+1)z+i\mathcal A} \prod\nolimits_{t=0}^r(e^{-2iz}-e^{-2ia_t})
\), whence
\[
c_{2t-r-1}=
(-2i)^{-r-1}(-1)^t\sum_{0\le j_1<\cdots< j_{t}\le r}
e^{- 2i(a_{j_1}+\cdots+a_{j_{t}}) + i\mathcal A},
\]
and, therefore,
\[
c_{\pm(r+1-2t)}=
(-1)^t (\pm 2i)^{-r-1}\sum_{0\le j_1<\cdots< j_{t}\le r}
e^{\pm 2i(a_{j_1}+\cdots+a_{j_{t}}) \mp i\mathcal A}.
\]
The last formula holds for~$t$ running over~$0,1,\dots,r+1$, although we actually need at most a half of this range --- as each of the coefficients appears twice.

On the other hand,~\eqref{eq:FE-defined} yields
\[
    F_{\varkappa-k}^\pm=(\pm 2i)^{n+1}
    \sum_{m=0}^{\lfloor{k}/{2}\rfloor} c_{\pm (r+1+2m-k)}
    \sum_{0\le k_1\le \cdots \le k_m\le n}\!\! e^{\pm i\left(\mathcal B + 2(b_{k_1}+\dots+b_{k_m})\right)},
\]
whence $F_{\varkappa-2t-1}^\pm=0$  for~$0\le 2t\le \varkappa-1$ (since $c_{\pm (r+1+2m-2t-1)}=0$), as well as
\[
F_{\varkappa-2t}^\pm
=2^{-\varkappa}\sum_{m=0}^t (-1)^{t-m}\!\!\!\!
\sum_{\substack{0\le j_1<\cdots< j_{t-m}\le r \\[2pt] 0\le k_1\le\cdots\le k_m\le n}}
\!\!\!\!
(\pm i)^{-\varkappa}
e^{\pm i\left(\mathcal B-\mathcal A+2(b_{k_1}+\cdots+b_{k_m} + a_{j_1}+\cdots+a_{j_{t-m}})\right)}
\]
for~$0\le 2t\le \varkappa$. The latter implies that
\[
F_{\varkappa-2t}^\pm e^{\pm i(\varkappa-2t)z}
=
2^{-\varkappa}\sum_{m=0}^t (-1)^{t-m}\!\!\!\!
\sum_{\substack{0\le j_1<\cdots< j_{t-m}\le r \\[2pt] 0\le k_1\le\cdots\le k_m\le n}}
\!\!\!\!
e^{\pm i\left((\varkappa-2t)z+ \mathcal B-\mathcal A +2(b_{k_1}+\cdots+b_{k_m} + a_{j_1}+\cdots+a_{j_{t-m}})- \varkappa\pi/2\right)},
\]
so the required formula for \(F_{\varkappa-2t}^+ e^{i(\varkappa-2t)z} + F_{\varkappa-2t}^- e^{-i(\varkappa-2t)z} \) holds.
Now, the corollary (except for the last assertion) follows by substituting this formula into
\eqref{eq:gen.trig.poly.pm.e} for even~$n$ or~\eqref{eq:gen.trig.poly.pm.o} for odd~$n$.
To obtain the last assertion, it is enough to calculate the difference of the plus and minus
cases of formula~\eqref{eq:sine.over.sine-odd.kappa} with~$\varkappa=2\lambda$.
\end{proof}
Note that, in particular, Corollary~\ref{cr:sine.over.sine} furnishes an expansion for the ratio of the form 
$$
\frac{\sin^p(z-\alpha)\cos^q(z-\gamma)}{\sin(z-\b)}
$$
by taking $r=p+q$, $a_1=\cdots= a_p=\alpha$, $a_{p+1}=\cdots=a_{r}=\gamma+\pi/2$ (cf.~\cite[eq.~(21)--(22)]{Chu}).

\begin{corollary}
    Suppose~$T_c(z)=\frac{1}{2}c_0+\sum_{t=0}^m c_t\cos(tz)$.
    Put~$\nu_{t,n}=\sign(t)\cdot\big(t+n-2\lfloor(t+n)/2\rfloor\big)$ as before, and retain the
    definition of $F^{\pm}_k$ from \eqref{eq:FE-defined} setting~$c_{-t}:=c_t$. Then
    \begin{equation*}
        \frac{T_c(z)}{\sin(z-\b)}
        =\frac{1}{2}\sum\nolimits_{k=1}^{m-n-1}\big(e^{ikz}F_k^{+}+e^{-ikz}F_k^{-}\big)
        \\
        +\sum_{k=0}^n\frac{\sum_{t=-m}^m 
            c_t \cos\big(tb_k+\nu_{t,n}(z-b_k)\big) }
        {2\sin(b_k-\b_{[k]})\sin(z-b_k)}
        .
    \end{equation*}

    Analogously, let~$T_s(z)=\sum_{t=1}^m c_t\sin(tz)$; then, for~$F^{\pm}_k$ as
    in~\eqref{eq:FE-defined} with~$c_{-t}:=-c_t$, the following is true:
    \begin{equation*}
        \frac{T_s(z)}{\sin(z-\b)}
        =\frac{1}{2i}\sum\nolimits_{k=1}^{m-n-1}\big(e^{ikz}F_k^{+}+e^{-ikz}F_k^{-}\big)
        +\sum_{k=0}^n\frac{\sum_{t=-m}^m 
            c_t \sin\big(tb_k+\nu_{t,n}(z-b_k)\big) }
        {2\sin(b_k-\b_{[k]})\sin(z-b_k)}.
    \end{equation*}
\end{corollary}
\begin{proof}
Define
$$
T(z)=\sum_{t=-m}^{m}c_te^{itz},~~\text{so that}~T_c(z)=\frac{1}{4}(T(z)+T(-z)),~~~T_s=\frac{1}{4i}(T(z)-T(-z)).
$$
Observe that~$E_j^\pm$ after replacing~$b_k\leftrightarrow-b_k$ becomes~$(-1)^{n+1}E_j^\mp$, and hence formula~\eqref{eq:gen.trig.poly.m} on exchanging~$z\leftrightarrow-z$ and~$b_k\leftrightarrow-b_k$ turns into
    \begin{multline*}
    \frac{(-1)^{n+1}T(-z)}{\sin(z-\b)}
    \!=\!(-1)^{n+1}\sum\nolimits_{k=1}^{m-n-1}\sum\nolimits_{j=0}^{\lfloor\frac{m-n-k-1}{2}\rfloor}\left(c_{2j+k+n+1}E_j^- e^{-ikz} + c_{-(2j+k+n+1)}E_j^+ e^{ikz}\right)
    \\
    +\sum_{k=0}^n\frac{(-1)^{n+1}\sum_{t=-m}^{m}c_t e^{-itb_k-i\nu_{t,n}(z-b_k)}}
    {\sin(b_k-\b_{[k]})\sin(z-b_k)}.
    \end{multline*}
    The last identity divided by $(-1)^{n+1}$ and added to~\eqref{eq:gen.trig.poly.m} yields
    \begin{multline*}
    \frac{T(z)+T(-z)}{\sin(z-\b)}
    =\sum\nolimits_{k=1}^{m-n-1} \sum\nolimits_{j=0}^{\lfloor\frac{m-n-k-1}{2}\rfloor}(c_{2j+k+n+1}+c_{-(2j+k+n+1)})
    (E_j^+ e^{ikz} + E_j^- e^{-ikz})
    \\
    +
    2\sum_{k=0}^n\frac{\sum_{t=-m}^m 
        c_t \cos\big(tb_k+\nu_{t,n}(z-b_k)\big) }
    {\sin(b_k-\b_{[k]})\sin(z-b_k)}.
    \end{multline*}
    As~$c_{-t}=c_{t}$, the last identity divided by $4$ leads to the first assertion of the corollary. In a similar way, one obtains
    \begin{multline*}
    \frac{T(z)-T(-z)}{\sin(z-\b)}
    \!=\!\sum\nolimits_{k=1}^{m-n-1} \sum\nolimits_{j=0}^{\lfloor\frac{m-n-k-1}{2}\rfloor}(c_{2j+k+n+1}-c_{-(2j+k+n+1)})
    (E_j^+ e^{ikz} - E_j^- e^{-ikz})
    \\
    +    2i\sum_{k=0}^n\frac{\sum_{t=-m}^m 
        c_t \sin\big(tb_k+\nu_{t,n}(z-b_k)\big) }
    {\sin(b_k-\b_{[k]})\sin(z-b_k)}.
    \end{multline*}
    So, recalling that $c_{t}=-c_{-t}$ and dividing by $4i$ renders the proof of the second assertion.
\end{proof}

\end{document}